\ifdefined\pdfoutput\pdfoutput=1\fi 
\documentclass[12pt,reqno]{amsart}

\usepackage[margin=1in]{geometry}

\usepackage{booktabs,array,microtype}
\usepackage[T1]{fontenc}
\usepackage{lmodern}
\usepackage{amsmath,amssymb,amsthm,mathtools}
\usepackage{enumitem,tikz}
\definecolor{figureBlue}{HTML}{2A78D6}
\definecolor{figureGreen}{HTML}{149A6B}
\definecolor{figureInk}{HTML}{6F6E68}
\definecolor{figurePale}{HTML}{8794A8}
\definecolor{figureRim}{HTML}{5B6B84}
\definecolor{pathBlue}{HTML}{14307A}
\pgfdeclareradialshading{glassSphere}{\pgfqpoint{-2.5bp}{3.5bp}}{color(0bp)=(white); color(12bp)=(white); color(18bp)=(figurePale!8); color(22bp)=(figurePale!20); color(24.3bp)=(figureRim!30); color(26bp)=(figureRim!40); color(50bp)=(figureRim!40)}
\colorlet{beadcolor}{figureBlue}
\pgfdeclareradialshading[beadcolor]{bead}{\pgfqpoint{-7bp}{8bp}}{color(0bp)=(beadcolor!25!white); color(8bp)=(beadcolor!80!white); color(17bp)=(beadcolor); color(25bp)=(beadcolor!72!black); color(50bp)=(beadcolor!72!black)}
\tikzset{bead/.style={shading=bead,/utils/exec={\colorlet{beadcolor}{#1}}}}
\usepackage[hidelinks]{hyperref}
\DeclareUrlCommand\path{\urlstyle{tt}\color{pathBlue}}

\makeatletter
\usepackage{iftex}
\ifPDFTeX
  \usepackage[T1]{fontenc}
  \usepackage[utf8]{inputenc}

  \usepackage{amsmath,amssymb}
  \usepackage{graphicx}
  \usepackage{xcolor}
  \usepackage{pdfrender}
  \usepackage[scaled=0.783]{DejaVuSansMono}

  \newsavebox{\lean@symbolbox}
  \newcommand{\leanmath}[1]{%
    \ifmmode #1\else
      \sbox{\lean@symbolbox}{\scalebox{0.8}{%
        \textpdfrender{TextRenderingMode=FillStroke,LineWidth=0.1pt}{%
          $\m@th\textstyle #1$}}}%
      \makebox[\fontdimen2\font][c]{%
        \ifdim\wd\lean@symbolbox>\fontdimen2\font
          \resizebox{\fontdimen2\font}{\height}{\usebox{\lean@symbolbox}}%
        \else
          \usebox{\lean@symbolbox}%
        \fi}%
    \fi}
  \newcommand{\leanoperator}[1]{%
    \leanmath{{\color[RGB]{24,45,75}\boldsymbol{#1}}}}

  \DeclareUnicodeCharacter{2115}{\leanmath{\mathbb{N}}}
  \DeclareUnicodeCharacter{2102}{\leanmath{\mathbb{C}}}
  \DeclareUnicodeCharacter{211D}{\leanmath{\mathbb{R}}}
  \DeclareUnicodeCharacter{2192}{\leanoperator{\to}}
  \DeclareUnicodeCharacter{2203}{\leanoperator{\exists}}
  \DeclareUnicodeCharacter{2200}{\leanoperator{\forall}}
  \DeclareUnicodeCharacter{2227}{\leanoperator{\wedge}}
  \DeclareUnicodeCharacter{2208}{\leanoperator{\in}}
  \DeclareUnicodeCharacter{2211}{\leanoperator{\sum}}
  \DeclareUnicodeCharacter{220F}{\leanoperator{\prod}}
  \DeclareUnicodeCharacter{03A6}{\leanmath{\Phi}}
  \DeclareUnicodeCharacter{03B3}{\leanmath{\gamma}}
  \DeclareUnicodeCharacter{03B5}{\leanmath{\varepsilon}}
  \DeclareUnicodeCharacter{03B8}{\leanmath{\theta}}
  \DeclareUnicodeCharacter{2016}{\leanoperator{\Vert}}
  \DeclareUnicodeCharacter{2260}{\leanoperator{\ne}}
  \DeclareUnicodeCharacter{2264}{\leanoperator{\le}}
  \DeclareUnicodeCharacter{2022}{\leanoperator{\bullet}}
  \DeclareUnicodeCharacter{207B}{\leanmath{{}^{-}}}
  \DeclareUnicodeCharacter{00B9}{\leanmath{{}^{1}}}
  \DeclareUnicodeCharacter{00AC}{\leanoperator{\neg}}
  \DeclareUnicodeCharacter{00D7}{\leanoperator{\times}}
  \DeclareUnicodeCharacter{02E2}{\leanmath{{}^{\mathrm{s}}}}
  \DeclareUnicodeCharacter{0393}{\leanmath{\Gamma}}
  \DeclareUnicodeCharacter{03A9}{\leanmath{\Omega}}
  \DeclareUnicodeCharacter{03C9}{\leanmath{\omega}}
  \DeclareUnicodeCharacter{1D62}{\leanmath{{}_{\mathrm{i}}}}
  \DeclareUnicodeCharacter{2080}{\leanmath{{}_{0}}}
  \DeclareUnicodeCharacter{2097}{\leanmath{{}_{\mathrm{l}}}}
  \DeclareUnicodeCharacter{2194}{\leanoperator{\leftrightarrow}}
  \DeclareUnicodeCharacter{221E}{\leanoperator{\infty}}
  \DeclareUnicodeCharacter{2228}{\leanoperator{\vee}}
  \DeclareUnicodeCharacter{2229}{\leanoperator{\cap}}
  \DeclareUnicodeCharacter{2243}{\leanoperator{\simeq}}
  \DeclareUnicodeCharacter{2265}{\leanoperator{\ge}}
  \DeclareUnicodeCharacter{2286}{\leanoperator{\subseteq}}
  \DeclareUnicodeCharacter{27E8}{\leanmath{\langle}}
  \DeclareUnicodeCharacter{27E9}{\leanmath{\rangle}}
\else
  \usepackage[no-math]{fontspec}
\fi

\makeatother
\usepackage{fvextra}
\usepackage{float}
\makeatletter
\def\PYG@reset{\let\PYG@it=\relax \let\PYG@bf=\relax%
    \let\PYG@ul=\relax \let\PYG@tc=\relax%
    \let\PYG@bc=\relax \let\PYG@ff=\relax}
\def\PYG@tok#1{\csname PYG@tok@#1\endcsname}
\def\PYG@toks#1+{\ifx\relax#1\empty\else%
    \PYG@tok{#1}\expandafter\PYG@toks\fi}
\def\PYG@do#1{\PYG@bc{\PYG@tc{\PYG@ul{%
    \PYG@it{\PYG@bf{\PYG@ff{#1}}}}}}}
\def\PYG#1#2{\PYG@reset\PYG@toks#1+\relax+\PYG@do{#2}}
\@namedef{PYG@tok@w}{\def\PYG@tc##1{\textcolor[rgb]{0.97,0.97,0.97}{##1}}}
\@namedef{PYG@tok@k}{\let\PYG@bf=\textbf\def\PYG@tc##1{\textcolor[rgb]{0.13,0.29,0.53}{##1}}}
\@namedef{PYG@tok@kn}{\let\PYG@bf=\textbf\def\PYG@tc##1{\textcolor[rgb]{0.13,0.29,0.53}{##1}}}
\@namedef{PYG@tok@kt}{\let\PYG@bf=\textbf\def\PYG@tc##1{\textcolor[rgb]{0.13,0.29,0.53}{##1}}}
\@namedef{PYG@tok@o}{\let\PYG@bf=\textbf\def\PYG@tc##1{\textcolor[rgb]{0.81,0.36,0.00}{##1}}}
\@namedef{PYG@tok@n}{\def\PYG@tc##1{\textcolor[rgb]{0.00,0.00,0.00}{##1}}}
\@namedef{PYG@tok@nb}{\def\PYG@tc##1{\textcolor[rgb]{0.13,0.29,0.53}{##1}}}
\@namedef{PYG@tok@bp}{\def\PYG@tc##1{\textcolor[rgb]{0.20,0.40,0.64}{##1}}}
\@namedef{PYG@tok@l}{\def\PYG@tc##1{\textcolor[rgb]{0.00,0.00,0.00}{##1}}}
\@namedef{PYG@tok@m}{\let\PYG@bf=\textbf\def\PYG@tc##1{\textcolor[rgb]{0.00,0.00,0.81}{##1}}}
\@namedef{PYG@tok@mi}{\let\PYG@bf=\textbf\def\PYG@tc##1{\textcolor[rgb]{0.00,0.00,0.81}{##1}}}
\def\PYGZus{\char`\_}
\def\PYGZca{\char`\^}
\def\PYGZlt{\char`\<}
\def\PYGZgt{\char`\>}
\def\PYGZhy{\char`\-}
\makeatother
\DefineVerbatimEnvironment{leancode}{Verbatim}{commandchars=\\\{\},breaklines,fontsize=\normalsize}
\newcommand{\lean}{\Verb[commandchars=\\\{\}]}

\ifPDFTeX
  \DeclareUnicodeCharacter{03C3}{\leanmath{\sigma}}
  \DeclareUnicodeCharacter{2081}{\leanmath{{}_1}}
  \DeclareUnicodeCharacter{2082}{\leanmath{{}_2}}
  \DeclareUnicodeCharacter{221A}{\leanoperator{\surd}}
\fi

\hypersetup{pdftitle={Energy minimization for eight points on the sphere},
 pdfauthor={Liudmyla Kryvonos, Lukas Liehr, Mitchell A. Taylor},
 pdfsubject={Riesz energy minimization}}

\mathtoolsset{showonlyrefs}

\definecolor{caGreen}{HTML}{0B5D1E}
\newcommand{\ca}[1]{#1}
\newcommand{\cacolor}{}

\newtheorem{theorem}{Theorem}[section]
\newtheorem{proposition}[theorem]{Proposition}
\newtheorem{lemma}[theorem]{Lemma}
\newtheorem{corollary}[theorem]{Corollary}
\theoremstyle{definition}
\newtheorem{definition}[theorem]{Definition}
\newtheorem{remark}[theorem]{Remark}
\numberwithin{equation}{section}
\newcommand{\R}{\mathbb R}
\newcommand{\Sph}{\mathbb S^2}
\newcommand{\Eh}{\widehat E}
\newcommand{\Sym}{\operatorname{Sym}}
\newcommand{\tr}{\operatorname{tr}}
\newcommand{\norm}[1]{\lVert#1\rVert}
\newcommand{\ip}[2]{\langle#1,#2\rangle}
\newcommand{\dd}{\mathrm d}
\newcommand{\doi}[1]{\href{https://doi.org/#1}{\nolinkurl{#1}}}
\setlist[enumerate]{label=\textnormal{(\roman*)},leftmargin=2em}
\title{Energy minimization for eight points on the sphere}

\author[Liudmyla Kryvonos]{Liudmyla Kryvonos}
\address{Department of Mathematics \& Statistics,
	University of North Florida, Jacksonville, FL 32224, USA}
\email{liudmyla.kryvonos@unf.edu}

\author[Lukas Liehr]{Lukas Liehr}
\address{Department of Mathematics, Bar-Ilan University, Ramat-Gan 5290002, Israel}
\email{lukas.liehr@biu.ac.il}

\author[Mitchell~A.~Taylor]{Mitchell A. Taylor}
\address{Mathematical Institute, University of Oxford, Andrew Wiles Building, Radcliffe Observatory Quarter, Woodstock Road, Oxford, OX2 6GG, United Kingdom}
\email{mitchell.taylor@maths.ox.ac.uk}

\date{}

\begin{document}

\begin{abstract}
We study the energy minimization problem for eight points on the unit sphere. For every $s$ in the range $-2<s<\infty$, we show that the unique global minimizer up to congruence of the Riesz $s$-energy (resp.~the logarithmic energy when $s=0$) is a square antiprism whose height is the unique solution of an explicit stationarity equation. The proof is computer-assisted and fully verified in Lean. We also give examples of energies arising from completely monotonic potentials for which no square antiprism is a global minimizer, answering in the negative a conjecture of Cohn and Woo.
\end{abstract}

\maketitle

\section{Introduction and results}\label{sec:introduction}
For the sphere $\Sph=\{x\in\R^3:\norm{x}=1\}$ and $N$ points $Y=(y_1,\ldots,y_N)\in(\Sph)^N$, we consider the discrete Riesz $s$-energies $E_s(Y)$ defined by
\begin{equation}\label{eq:energy}
 E_s(Y)=\sum_{1\le i<j\le N}K_s(y_i,y_j),\quad
 K_s(x,y)=
 \begin{cases}
 \frac1s\norm{x-y}^{-s},&s\ne 0,\\
 -\log\norm{x-y},&s=0,
 \end{cases}
\end{equation}
where $\log$ denotes the natural logarithm and $\norm{\cdot}$ denotes the Euclidean norm. On the diagonal, we set $K_s(x,x)=0$ if $s<0$ and $K_s(x,x)=+\infty$ if $s\geq0$, which extends $K_s$ continuously. The classical \textit{Thomson problem}, corresponding to the case $s=1$, is to determine the configuration of $N$ points on $\Sph$ that minimizes the Coulomb energy \cite{Tho04}. This naturally leads to the more general problem of minimizing the Riesz $s$-energy, i.e., computing the quantity
$$
\mathcal E_s(N)=\min_{Y\in(\Sph)^N}E_s(Y).
$$
Note that $\mathcal E_s(N)$ exists by lower semicontinuity and compactness.

Our main results concern the eight-point problem for every $s >-2$. Recall that the cases $s\leq-2$ are covered by Bj\"orck's results \cite[Theorem~5]{Bjo56} and \cite[Theorem~7]{Bjo56}: at $s=-2$, every configuration with zero centroid is minimizing, while for $s<-2$, the minimizers consist of two antipodal points.

Table \ref{tab:smallN} summarizes the Riesz cases relevant here, together with selected references, some of which concern only particular values of $s$. The configurations for $N\in\{2,3,4,6,12\}$ fall within the general theory of universal optimality developed by Cohn and Kumar \cite{CohK07}, and earlier results include the Coulomb bounds of Yudin \cite{Yud93}, the logarithmic octahedron result of Kolushov and Yudin \cite{KolY97}, and Andreev's work on the icosahedron \cite{And96}. In particular, these configurations minimize the Riesz $s$-energy for every $s>-2$. The five-point problem shows that this independence of the parameter $s$ is a special property. For five points, Dragnev, Legg, and Townsend \cite{DLT02} proved that the triangular bipyramid is, up to orthogonal transformations, the unique logarithmic minimizer, and Schwartz \cite{Sch13} established the Coulomb case by a computer-assisted proof. His subsequent work \cite{Sch16} identifies the phase-transition value $S_{5} \approx 15.048$: the triangular bipyramid is optimal for $s \in (-2,0)\cup(0, S_{5}]$, whereas a square-based pyramid becomes optimal immediately beyond this value (see also \cite{Sch23}).

\begin{table}[!htbp]
 \centering
 \begin{minipage}{0.78\textwidth}
 \centering
 \small
 \renewcommand{\arraystretch}{1.15}
 \begin{tabular}{@{}clll@{}}
  \toprule
  $N$ & Configuration & Optimality Range & References\\
  \midrule
  $2$ & Antipodal pair & $s > -2$ & \cite{CohK07}\\
  $3$ & Equilateral triangle & $s > -2$ & \cite{CohK07}\\
  $4$ & Regular tetrahedron & $s > -2$ & \cite{CohK07,KolY97,Yud93}\\
  $5$ & Triangular bipyramid & $-2 < s\leq S_5$ & \cite{DLT02,Sch13,Sch16}\\
  $6$ & Regular octahedron & $s > -2 $ & \cite{CohK07,KolY97,Yud93}\\
  $12$ & Regular icosahedron & $s > -2$ & \cite{And96,CohK07}\\
  \bottomrule
 \end{tabular}
 \vspace{0.3cm}
 \caption{Known globally minimizing configurations.
 }
 \label{tab:smallN}
 \end{minipage}
\end{table}

Beyond the cases discussed above, identifying globally minimizing configurations remains challenging even for relatively small values of $N$. Numerical studies have provided further insight into the Thomson problem by identifying candidate minimizers and exploring the landscape of competing local minima \cite{AmoreEtAl25,ErbHoc91,MDH96,RSZ94,RSZ95,WU06}. In parallel, efficient methods for constructing well-distributed point sets and finding low-energy configurations have been developed and compared in \cite{AmoreWales26,HMS16,MDH96}. For broader surveys of discrete energy minimization and point distributions on spheres, we refer to \cite{BraGra15,SafKui97}.

\subsection{Optimality of the square antiprism}

For the case $N=8$, Cohn and Woo \cite[Conjecture~14]{CW} conjectured that every completely monotonic potential of the squared distance is minimized by a square antiprism, with height depending on the potential. They found numerical evidence supporting their conjecture, but converting these numerical results into a rigorous proof is delicate. In the present paper, we prove this conclusion, with uniqueness up to congruence, for every Riesz $s$-energy with $s >-2$. On the other hand, we show that the conjecture is not true in general, i.e., there exist completely monotonic potentials for which no square antiprism is a global minimizer of the induced energy.

A square antiprism consists of two squares lying in parallel planes, with one rotated by $\pi/4$ relative to the other. Specifically, for $0<a<1$, we set $r_a=\sqrt{1-a}$, $h_a=\sqrt a$ and define the square antiprism $X(a)$ by
\begin{equation}\label{eq:family}
\begin{split}
 X(a)={}&\left\{\left(r_a\cos\frac{j\pi}{2},
 r_a\sin\frac{j\pi}{2},h_a\right):0\le j\le3\right\}\\
 &\cup\left\{\left(r_a\cos\left(\frac\pi4+\frac{j\pi}{2}\right),
 r_a\sin\left(\frac\pi4+\frac{j\pi}{2}\right),-h_a\right):0\le j\le3\right\},
\end{split}
\end{equation}
where $a$ is the squared half-height of the antiprism. We regard $X(a)=(x_1,\ldots,x_8)$ as an element of $(\Sph)^8$ by labeling its vertices in the displayed order. Figure~\ref{fig:antiprism} shows the configuration and the relative rotation of its two squares.

\begin{figure}[htbp]
\centering
\begin{tikzpicture}[line cap=round,line join=round,font=\small]
\path[use as bounding box] (-6.6,-3.25) rectangle (6.6,3.45);
\node at (-3.65,3.18) {\textnormal{(a)} Perspective};
\node at (3.65,3.18) {\textnormal{(b)} View from above};
\shade[inner color=black!22,outer color=white] (-3.650000,-2.930000) ellipse[x radius=1.850000cm,y radius=0.270000cm];
\shade[shading=glassSphere] (-3.650000,0.000000) circle[radius=2.550000cm];
\draw[figureRim!70,line width=0.4pt] (-3.650000,0.000000) circle[radius=2.550000cm];
\draw[figureInk!60,line width=0.55pt,dash pattern=on 2.4pt off 2pt] (-3.650000,0.000000) ellipse[x radius=2.550000cm,y radius=0.830199cm];
\draw[figureGreen!65,line width=0.55pt] (-3.650000,-1.351254) ellipse[x radius=2.111903cm,y radius=0.687568cm];
\fill[figureGreen!7] (-1.564098,-1.243694) -- (-3.980374,-0.672150) -- (-5.735902,-1.458813) -- (-3.319626,-2.030357) -- cycle;
\draw[figureBlue!65,line width=0.55pt] (-3.650000,1.351254) ellipse[x radius=2.111903cm,y radius=0.687568cm];
\fill[figureBlue!7] (-1.941435,0.947111) -- (-2.408655,1.907508) -- (-5.358565,1.755396) -- (-4.891345,0.794999) -- cycle;
\draw[white,line width=1.85pt] (-2.408655,1.907508) -- (-3.980374,-0.672150);
\draw[figureInk!42,line width=1.05pt] (-2.408655,1.907508) -- (-3.980374,-0.672150);
\draw[white,line width=1.80pt] (-1.564098,-1.243694) -- (-3.980374,-0.672150);
\draw[figureGreen!42,line width=1.10pt] (-1.564098,-1.243694) -- (-3.980374,-0.672150);
\draw[white,line width=1.85pt] (-3.980374,-0.672150) -- (-5.358565,1.755396);
\draw[figureInk!42,line width=1.05pt] (-3.980374,-0.672150) -- (-5.358565,1.755396);
\draw[white,line width=1.80pt] (-3.980374,-0.672150) -- (-5.735902,-1.458813);
\draw[figureGreen!42,line width=1.10pt] (-3.980374,-0.672150) -- (-5.735902,-1.458813);
\draw[white,line width=1.85pt] (-1.564098,-1.243694) -- (-2.408655,1.907508);
\draw[figureInk!42,line width=1.05pt] (-1.564098,-1.243694) -- (-2.408655,1.907508);
\draw[white,line width=1.80pt] (-2.408655,1.907508) -- (-5.358565,1.755396);
\draw[figureBlue!42,line width=1.10pt] (-2.408655,1.907508) -- (-5.358565,1.755396);
\draw[white,line width=1.85pt] (-5.358565,1.755396) -- (-5.735902,-1.458813);
\draw[figureInk!42,line width=1.05pt] (-5.358565,1.755396) -- (-5.735902,-1.458813);
\draw[white,line width=1.80pt] (-1.941435,0.947111) -- (-2.408655,1.907508);
\draw[figureBlue!90,line width=1.10pt] (-1.941435,0.947111) -- (-2.408655,1.907508);
\draw[white,line width=1.80pt] (-3.319626,-2.030357) -- (-1.564098,-1.243694);
\draw[figureGreen!90,line width=1.10pt] (-3.319626,-2.030357) -- (-1.564098,-1.243694);
\draw[white,line width=1.85pt] (-1.941435,0.947111) -- (-1.564098,-1.243694);
\draw[figureInk!90,line width=1.05pt] (-1.941435,0.947111) -- (-1.564098,-1.243694);
\draw[white,line width=1.80pt] (-5.735902,-1.458813) -- (-3.319626,-2.030357);
\draw[figureGreen!90,line width=1.10pt] (-5.735902,-1.458813) -- (-3.319626,-2.030357);
\draw[white,line width=1.80pt] (-5.358565,1.755396) -- (-4.891345,0.794999);
\draw[figureBlue!90,line width=1.10pt] (-5.358565,1.755396) -- (-4.891345,0.794999);
\draw[white,line width=1.85pt] (-5.735902,-1.458813) -- (-4.891345,0.794999);
\draw[figureInk!90,line width=1.05pt] (-5.735902,-1.458813) -- (-4.891345,0.794999);
\draw[white,line width=1.85pt] (-3.319626,-2.030357) -- (-1.941435,0.947111);
\draw[figureInk!90,line width=1.05pt] (-3.319626,-2.030357) -- (-1.941435,0.947111);
\draw[white,line width=1.85pt] (-4.891345,0.794999) -- (-3.319626,-2.030357);
\draw[figureInk!90,line width=1.05pt] (-4.891345,0.794999) -- (-3.319626,-2.030357);
\draw[white,line width=1.80pt] (-4.891345,0.794999) -- (-1.941435,0.947111);
\draw[figureBlue!90,line width=1.10pt] (-4.891345,0.794999) -- (-1.941435,0.947111);
\fill[white] (-3.980374,-0.672150) circle[radius=3.3pt];
\shade[bead=figureGreen!65] (-3.980374,-0.672150) circle[radius=2.8pt];
\fill[white] (-2.408655,1.907508) circle[radius=3.3pt];
\shade[bead=figureBlue!65] (-2.408655,1.907508) circle[radius=2.8pt];
\fill[white] (-1.564098,-1.243694) circle[radius=3.3pt];
\shade[bead=figureGreen!65] (-1.564098,-1.243694) circle[radius=2.8pt];
\fill[white] (-5.358565,1.755396) circle[radius=3.3pt];
\shade[bead=figureBlue!65] (-5.358565,1.755396) circle[radius=2.8pt];
\fill[white] (-5.735902,-1.458813) circle[radius=3.3pt];
\shade[bead=figureGreen!65] (-5.735902,-1.458813) circle[radius=2.8pt];
\fill[white] (-3.319626,-2.030357) circle[radius=3.3pt];
\shade[bead=figureGreen!100] (-3.319626,-2.030357) circle[radius=2.8pt];
\fill[white] (-1.941435,0.947111) circle[radius=3.3pt];
\shade[bead=figureBlue!100] (-1.941435,0.947111) circle[radius=2.8pt];
\fill[white] (-4.891345,0.794999) circle[radius=3.3pt];
\shade[bead=figureBlue!100] (-4.891345,0.794999) circle[radius=2.8pt];
\shade[shading=glassSphere] (3.650000,0.000000) circle[radius=2.550000cm];
\draw[figureRim!70,line width=0.4pt] (3.650000,0.000000) circle[radius=2.550000cm];
\draw[figureInk!60,line width=0.55pt,dash pattern=on 2.4pt off 2pt] (3.650000,0.000000) circle[radius=2.111903cm];
\path[fill=figureGreen!5] (5.143341,1.493341) -- (2.156659,1.493341) -- (2.156659,-1.493341) -- (5.143341,-1.493341) -- cycle;
\path[fill=figureBlue!5] (5.761903,0.000000) -- (3.650000,2.111903) -- (1.538097,0.000000) -- (3.650000,-2.111903) -- cycle;
\draw[figureGreen,line width=1.1pt] (5.143341,1.493341) -- (2.156659,1.493341) -- (2.156659,-1.493341) -- (5.143341,-1.493341) -- cycle;
\draw[figureBlue,line width=1.1pt] (5.761903,0.000000) -- (3.650000,2.111903) -- (1.538097,0.000000) -- (3.650000,-2.111903) -- cycle;
\draw[figureInk!58,line width=0.45pt] (3.650000,0.000000) -- (5.761903,0.000000);
\draw[figureInk!58,line width=0.45pt] (3.650000,0.000000) -- (5.143341,1.493341);
\draw[figureInk!90,line width=0.65pt,->,>=stealth] (4.310000,0.000000) arc[start angle=0,end angle=45,radius=0.660000cm];
\node[inner sep=1pt,font=\footnotesize] at (4.712461,0.440086) {$\pi/4$};
\fill[figureInk!65] (3.650000,0.000000) circle[radius=0.9pt];
\fill[white] (5.761903,0.000000) circle[radius=3.3pt];
\shade[bead=figureBlue] (5.761903,0.000000) circle[radius=2.8pt];
\fill[white] (5.143341,1.493341) circle[radius=3.3pt];
\shade[bead=figureGreen] (5.143341,1.493341) circle[radius=2.8pt];
\fill[white] (3.650000,2.111903) circle[radius=3.3pt];
\shade[bead=figureBlue] (3.650000,2.111903) circle[radius=2.8pt];
\fill[white] (2.156659,1.493341) circle[radius=3.3pt];
\shade[bead=figureGreen] (2.156659,1.493341) circle[radius=2.8pt];
\fill[white] (1.538097,0.000000) circle[radius=3.3pt];
\shade[bead=figureBlue] (1.538097,0.000000) circle[radius=2.8pt];
\fill[white] (2.156659,-1.493341) circle[radius=3.3pt];
\shade[bead=figureGreen] (2.156659,-1.493341) circle[radius=2.8pt];
\fill[white] (3.650000,-2.111903) circle[radius=3.3pt];
\shade[bead=figureBlue] (3.650000,-2.111903) circle[radius=2.8pt];
\fill[white] (5.143341,-1.493341) circle[radius=3.3pt];
\shade[bead=figureGreen] (5.143341,-1.493341) circle[radius=2.8pt];
\node[figureInk!95] at (-3.65,-3.00) {};
\node[figureInk!95] at (3.65,-3.00) {};
\end{tikzpicture}
\caption{The square antiprism $X(a)$ for the value of $a$ that minimizes the Coulomb energy $E_1(X(a))$. The upper square is blue and the lower square is green. The view from above shows their relative angle $\pi/4$. Gray segments in \textnormal{(a)} join nearest vertices on opposite squares.}
\label{fig:antiprism}
\vspace{0.4cm}
\end{figure}

Among the $28$ pairs of vertices of $X(a)$, we distinguish four types of distances: two within each square and two between the squares. We label these four types by the letters $A,B,C,D$. For each type $\alpha\in\{A,B,C,D\}$, all pairs of vertices whose distance is of type $\alpha$ have the same squared distance, which we denote by $d_\alpha$. The four types, the squared distances $d_A,d_B,d_C,d_D$ and the multiplicities, i.e., the numbers of pairs of each type, are recorded in Table~\ref{tab:distances} and visualized in Figure~\ref{fig:types}.

\begin{table}[htbp]
\centering
\small
\renewcommand{\arraystretch}{1.15}
\begin{tabular}{@{}cllc@{}}
\toprule
Type & Pairs of vertices & Squared distance & Multiplicity\\
\midrule
$A$ & adjacent vertices of the same square & $d_A=2(1-a)$ & $8$\\
$B$ & opposite vertices of the same square & $d_B=4(1-a)$ & $4$\\
$C$ & nearest vertices of different squares & $d_C=2-\sqrt2+(2+\sqrt2)a$ & $8$\\
$D$ & farthest vertices of different squares & $d_D=2+\sqrt2+(2-\sqrt2)a$ & $8$\\
\bottomrule
\end{tabular}
\vspace{0.3cm}
\caption{The four types of distances between the vertices of the square antiprism $X(a)$.}
\label{tab:distances}
\end{table}

\begin{figure}[htbp]
\centering
\begin{tikzpicture}[line cap=round,line join=round,font=\small]
\path[use as bounding box] (-3.0,-2.45) rectangle (7.2,2.45);
\draw[figureInk!32,line width=0.4pt,dash pattern=on 1.8pt off 2pt] (0,0) circle[radius=2cm];
\draw[figureBlue!30,line width=0.6pt] (2,0) -- (0,2) -- (-2,0) -- (0,-2) -- cycle;
\draw[figureGreen!40,line width=0.6pt] (1.414214,1.414214) -- (-1.414214,1.414214) -- (-1.414214,-1.414214) -- (1.414214,-1.414214) -- cycle;
\draw[figureBlue,line width=1.2pt] (2,0) -- (0,2);
\draw[figureBlue,line width=1.2pt] (2,0) -- (0,-2);
\draw[figureBlue,line width=1.2pt,dash pattern=on 4pt off 2.5pt] (2,0) -- (-2,0);
\draw[figureGreen,line width=1.2pt] (2,0) -- (1.414214,1.414214);
\draw[figureGreen,line width=1.2pt] (2,0) -- (1.414214,-1.414214);
\draw[figureGreen,line width=1.2pt,dash pattern=on 4pt off 2.5pt] (2,0) -- (-1.414214,1.414214);
\draw[figureGreen,line width=1.2pt,dash pattern=on 4pt off 2.5pt] (2,0) -- (-1.414214,-1.414214);
\filldraw[fill=figureBlue,draw=white,line width=0.6pt] (2,0) circle[radius=2.6pt];
\filldraw[fill=figureBlue,draw=white,line width=0.6pt] (0,2) circle[radius=2.6pt];
\filldraw[fill=figureBlue,draw=white,line width=0.6pt] (-2,0) circle[radius=2.6pt];
\filldraw[fill=figureBlue,draw=white,line width=0.6pt] (0,-2) circle[radius=2.6pt];
\filldraw[fill=figureGreen,draw=white,line width=0.6pt] (1.414214,1.414214) circle[radius=2.6pt];
\filldraw[fill=figureGreen,draw=white,line width=0.6pt] (-1.414214,1.414214) circle[radius=2.6pt];
\filldraw[fill=figureGreen,draw=white,line width=0.6pt] (-1.414214,-1.414214) circle[radius=2.6pt];
\filldraw[fill=figureGreen,draw=white,line width=0.6pt] (1.414214,-1.414214) circle[radius=2.6pt];
\node[anchor=west,inner sep=4pt] at (2,0) {$x_1$};
\node[inner sep=1pt] at (0.72,0.98) {$A$};
\node[inner sep=1pt] at (-0.75,0.2) {$B$};
\node[inner sep=1pt] at (2.0,0.82) {$C$};
\node[inner sep=1pt] at (-0.35,1.2) {$D$};
\node[anchor=west] at (3.6,1.55) {Within a square};
\draw[figureBlue,line width=1.2pt] (3.7,1.05) -- (4.5,1.05);
\node[anchor=west] at (4.7,1.05) {type $A$};
\draw[figureBlue,line width=1.2pt,dash pattern=on 4pt off 2.5pt] (3.7,0.55) -- (4.5,0.55);
\node[anchor=west] at (4.7,0.55) {type $B$};
\node[anchor=west] at (3.6,-0.35) {Between the squares};
\draw[figureGreen,line width=1.2pt] (3.7,-0.85) -- (4.5,-0.85);
\node[anchor=west] at (4.7,-0.85) {type $C$};
\draw[figureGreen,line width=1.2pt,dash pattern=on 4pt off 2.5pt] (3.7,-1.35) -- (4.5,-1.35);
\node[anchor=west] at (4.7,-1.35) {type $D$};
\end{tikzpicture}
\caption{The four types of distances in $X(a)$, viewed from above. The highlighted segments join the vertex $x_1$ of the upper square to the remaining seven vertices: two segments of type $A$, one of type $B$, two of type $C$ and two of type $D$.}
\label{fig:types}
\end{figure}

The critical points of $E_s$ within the family $\{X(a):a\in(0,1)\}$ are characterized by the stationarity equation $F(s,a)=0$, where $F$ is given by
\begin{equation}\label{eq:F}
 F(s,a)=16d_A^{-q_s}+16d_B^{-q_s}
       -8(2+\sqrt2)d_C^{-q_s}-8(2-\sqrt2)d_D^{-q_s},
 \quad q_s=1+s/2.
\end{equation}
We will show in Proposition~\ref{prop:parameter} that for every fixed $s>-2$, the equation $F(s,a)=0$ has a unique solution in the interval $(0,1)$, which motivates the following definition.

\begin{definition}\label{def:as}
For $s>-2$, we define the number $a_s$ as the unique solution of the equation $F(s,a)=0$ in the interval $(0,1)$.
\end{definition}

The first main result of this article settles the energy minimization problem of $8$ points on the sphere for every $s>-2$. To state it, we say that two configurations $Y=(y_1,\ldots,y_N)$ and $Z=(z_1,\ldots,z_N)$ in $(\Sph)^N$ are \emph{congruent} if there exist an orthogonal matrix $U\in\R^{3\times3}$ and a permutation $\sigma$ of $\{1,\ldots,N\}$ such that $y_i=Uz_{\sigma(i)}$ for every $i\in\{1,\ldots,N\}$.

\begin{theorem}\label{thm:global}
For $s > -2$ and every $Y\in(\Sph)^8$, one has $E_s(Y)\ge E_s(X(a_s))$, with equality if and only if $Y$ is congruent to $X(a_s)$.
\end{theorem}

The proof of Theorem \ref{thm:global} is computer-assisted and verified in Lean. The certificates and the Lean formalization are contained in the accompanying repository \cite{N8Fixed}. The main ideas of the analytic parts of the proof of Theorem \ref{thm:global} are discussed in Section \ref{sec:proof-guide}.

\subsection{Completely monotonic potentials}

Our next result shows that the global optimality of square antiprisms does not extend to all energies arising from completely monotonic potentials of the squared distance. Recall that a smooth function $v$ on an interval $I\subseteq(0,\infty)$ is called \emph{completely monotonic} if
$$
(-1)^k v^{(k)}(t)\geq0, \quad k\geq0, \quad t\in I,
$$
and \emph{strictly completely monotonic} if all of these inequalities are strict. For such a function $v$ and a configuration $Y$ of eight distinct points on $\Sph$, we write the corresponding $v$-energy as
\begin{equation}\label{eq:cm-energy}
 E_v(Y)=\sum_{1\leq i<j\leq8}v(\norm{y_i-y_j}^2).
\end{equation}

\begin{theorem}\label{thm:CM}
There exist a strictly completely monotonic function $v$ on $(0,\infty)$ and a configuration $Y$ of eight distinct points on $\Sph$ such that $E_v(Y)<E_v(X(a))$ for every $a\in(0,1)$. In particular, no square antiprism is a global minimizer of $E_v$.
\end{theorem}

Since the restriction of $v$ to $(0,4]$ is completely monotonic, Theorem~\ref{thm:CM} answers in the negative a universality conjecture of Cohn and Woo \cite[Conjecture~14]{CW}.

Cohn and Woo formulated a similar conjecture for ten points on the sphere $\mathbb S^3=\{x\in\R^4:\norm{x}=1\}$, which involves the following two configurations. Let $u_1,\ldots,u_5\in\mathbb S^3$ be the vertices of a regular simplex, i.e., $u_i\cdot u_j=-1/4$ for $i\neq j$. The \emph{Petersen code} $\mathcal P$ consists of the normalized midpoints of the ten edges of this simplex, i.e.,
$$
\mathcal P=\left\{\sqrt{\frac23}\,(u_i+u_j):1\leq i<j\leq5\right\}.
$$
Two points of $\mathcal P$ have inner product $-2/3$ if the corresponding edges are disjoint, and inner product $1/6$ otherwise. Hence, the pairs with inner product $-2/3$ form the edges of the so-called Petersen graph. The \emph{orthogonal pentagon code} $\mathcal O$ consists of the vertices of two regular pentagons that are inscribed in the unit circles of two orthogonal planes, i.e.,
$$
\mathcal O=\left\{\left(\cos\frac{2\pi j}5,\sin\frac{2\pi j}5,0,0\right):0\leq j\leq4\right\}\cup\left\{\left(0,0,\cos\frac{2\pi j}5,\sin\frac{2\pi j}5\right):0\leq j\leq4\right\}.
$$
Both configurations are shown in Figure~\ref{fig:codes}. It was conjectured in \cite[Conjecture~13]{CW} that for every completely monotonic function $v$ on $(0,4]$, the Petersen code or the orthogonal pentagon code minimizes the energy
$$
E_v(Y)=\sum_{1\leq i<j\leq10}v(\norm{y_i-y_j}^2)
$$
among all configurations $Y=(y_1,\ldots,y_{10})$ of ten distinct points on $\mathbb S^3$. The following theorem, which was communicated to us by Henry Cohn, disproves this conjecture.

\begin{figure}[htbp]
\centering
\begin{tikzpicture}[line cap=round,line join=round,font=\small]
\path[use as bounding box] (-6.6,-2.75) rectangle (6.6,3.25);
\node at (-3.65,2.95) {\textnormal{(a)} Petersen code};
\node at (3.65,2.95) {\textnormal{(b)} Orthogonal pentagon code};
\begin{scope}[shift={(-3.65,-0.15)}]
\foreach \k in {0,1,2,3,4} {
  \coordinate (o\k) at ({90+72*\k}:2.05);
  \coordinate (i\k) at ({90+72*\k}:0.95);
}
\draw[figureInk!85,line width=0.75pt] (o0) -- (o1) -- (o2) -- (o3) -- (o4) -- cycle;
\draw[figureInk!85,line width=0.75pt] (i0) -- (i2) -- (i4) -- (i1) -- (i3) -- cycle;
\foreach \k in {0,1,2,3,4} {
  \draw[figureInk!85,line width=0.75pt] (o\k) -- (i\k);
}
\foreach \k in {0,1,2,3,4} {
  \filldraw[fill=figureBlue,draw=white,line width=0.65pt] (o\k) circle[radius=3pt];
  \filldraw[fill=figureBlue,draw=white,line width=0.65pt] (i\k) circle[radius=3pt];
}
\node at ({90}:2.42) {$12$};
\node at ({162}:2.47) {$34$};
\node at ({234}:2.47) {$15$};
\node at ({306}:2.47) {$23$};
\node at ({18}:2.47) {$45$};
\node at ({90+19}:1.02) {$35$};
\node at ({162+19}:1.05) {$25$};
\node at ({234+19}:1.05) {$24$};
\node at ({306+19}:1.05) {$14$};
\node at ({18+19}:1.05) {$13$};
\end{scope}
\begin{scope}[shift={(2.05,-0.05)}]
\draw[figureRim!70,line width=0.45pt] (0,0) circle[radius=1.3cm];
\draw[figureBlue,line width=1.1pt] (90:1.3) -- (162:1.3) -- (234:1.3) -- (306:1.3) -- (18:1.3) -- cycle;
\foreach \k in {0,1,2,3,4} {
  \filldraw[fill=figureBlue,draw=white,line width=0.65pt] ({90+72*\k}:1.3) circle[radius=3pt];
}
\node at (0,-1.85) {$\R^2\times\{0\}$};
\end{scope}
\begin{scope}[shift={(5.25,-0.05)}]
\draw[figureRim!70,line width=0.45pt] (0,0) circle[radius=1.3cm];
\draw[figureGreen,line width=1.1pt] (90:1.3) -- (162:1.3) -- (234:1.3) -- (306:1.3) -- (18:1.3) -- cycle;
\foreach \k in {0,1,2,3,4} {
  \filldraw[fill=figureGreen,draw=white,line width=0.65pt] ({90+72*\k}:1.3) circle[radius=3pt];
}
\node at (0,-1.85) {$\{0\}\times\R^2$};
\end{scope}
\node at (3.65,-0.05) {$\perp$};
\end{tikzpicture}
\caption{The two configurations of ten points on $\mathbb S^3$ in Theorem~\ref{thm:CM2}. In (a), the vertex labeled $ij$ represents the point $\sqrt{2/3}\,(u_i+u_j)$ of the Petersen code, and the edges join the pairs of points with inner product $-2/3$. In (b), the two regular pentagons of the orthogonal pentagon code lie in the orthogonal planes $\R^2\times\{0\}$ and $\{0\}\times\R^2$ of $\R^4$.}
\label{fig:codes}
\end{figure}

\begin{theorem}[H. Cohn~\cite{CohnPC}]\label{thm:CM2}
There exists a completely monotonic function $v$ on $(0,4]$ such that neither the Petersen code nor the orthogonal pentagon code minimizes $E_v$ among all configurations of ten distinct points on $\mathbb S^3$.
\end{theorem}

\subsection{Computer-assisted proofs and Lean verification}

Theorems~\ref{thm:global} and \ref{thm:CM} have been fully verified in Lean. Moreover, the proof of Theorem~\ref{thm:global} is computer-assisted, and the Lean verification of the computer-assisted part of the proof is a crucial novelty of the paper. The computer-assisted statements are Propositions~\ref{cap:nonnegative}, \ref{cap:negative} and~\ref{cap:tail}. In the associated GitHub repository \cite{N8Fixed}, we include all the numerical certificates that were used to prove global optimality of the square antiprism for all $s > -2$. To make the Lean verification of the computer-assisted parts of the proofs computationally tractable, the original certificates were reduced and transformed into a form better suited for verification in Lean.

\subsection{Use of large language models} The collaboration between the three authors began at the Shanks Workshop 2026 at Vanderbilt University, and we thank the organizers for making this possible. Since then, the authors have discussed various problems on energy optimization and how to Lean verify such proofs.

The main ideas in the proofs of Theorem~\ref{thm:global} and Theorem~\ref{thm:CM} are due to GPT-6 Astra. The primary contributions of the authors were to understand and significantly simplify the proofs, streamline the mathematical arguments and numerical computations, and formally verify the numerical certificates in Lean. The Lean verification of the non-computer-assisted parts was relatively straightforward, in view of our previous experience with Lean verifying research level mathematics \cite{BFLT26,LTY26}.

The Lean verification of the computer-assisted parts for $s=0$ and $s=1$ was carried out with GPT-6 in Codex. Claude Opus 5.5 was used for the Lean verification of the numerical certificates for the full range $s>-2$. Completing this latter verification required a substantial reduction in the size of the original certificates. Indeed, the original proof for general Riesz energies was achieved only after extensive iterative refinements and required over $90$ pages of mathematical reductions and a huge computer-assisted component. Lean verifying such a proof would have likely been impossible without specialized equipment. Instead, we were able to find a way to reduce the computer-assisted part by over $99.9\%$, which eventually allowed us to complete the Lean verification for the whole range $s>-2$.

\section{Notation and proof strategy}\label{sec:proof-guide}

\subsection{Notation}

We write $x\cdot y$ for the Euclidean inner product of $x,y\in\R^3$. Since $x\cdot y=1-\norm{x-y}^2/2$ for $x,y\in\Sph$, two vertices of $X(a)$ whose distance is of type $\alpha\in\{A,B,C,D\}$ have the inner product $t_\alpha=1-d_\alpha/2$, i.e.,
\begin{equation}\label{eq:nodes}
 t_A=a,\quad t_B=2a-1,\quad t_C=-a+(1-a)/\sqrt2,\quad t_D=-a-(1-a)/\sqrt2.
\end{equation}
Further, for a polynomial $H$ in one variable and a configuration $Y=(y_1,\ldots,y_8)\in(\Sph)^8$, we write
$$
E_H(Y)=\sum_{i<j}H(y_i\cdot y_j).
$$

For a real symmetric matrix $M$, we write $M\succeq0$ if $M$ is positive semidefinite and $M\succ0$ if $M$ is positive definite. For real symmetric matrices $F$ and $R$ of the same dimensions, we denote by $\ip{F}{R}=\tr(FR)$ their trace inner product. Moreover, $J$ denotes the matrix with all entries equal to $1$, whose dimensions will be clear from the context, and $\delta_{k0}$ denotes the Kronecker delta.

\subsection{Proof strategy}

We now discuss the main steps of the proof of Theorem~\ref{thm:global}.

First, we show in Section~\ref{sec:geometry} that for every $s>-2$ there exists a unique minimizer $a_s$ of $a\mapsto E_s(X(a))$ on $(0,1)$, and that $X(a_s)$ is a critical point of $E_s$ on $(\Sph)^8$ (see Proposition~\ref{prop:parameter}).

In Section~\ref{sec:geometry}, we  write $E_s(Y)$, up to an additive constant, as the sum of $\phi_s(y_i\cdot y_j)$ over all pairs $i<j$, where $\phi_s$ is a function on $[-1,1)$. In Section~\ref{sec:minorants}, following an argument which is attributed to Yudin in \cite[Lemma~9]{CW} (see also \cite[Lemma~2.1]{CohK07} and \cite[Lemma~5.1]{CohK07}), we use Hermite interpolation to construct a polynomial $H_s$ such that $H_s\leq\phi_s$ on $[-1,1)$ and such that $H_s$ and $H_s'$ agree with $\phi_s$ and $\phi_s'$ at the four inner products $t_A,t_B,t_C,t_D$ of $X(a_s)$. Up to the same additive constant, the polynomial energy $E_{H_s}(Y)$ is then a lower bound for $E_s(Y)$ with equality for $Y=X(a_s)$, and it suffices to show that $X(a_s)$ minimizes $E_{H_s}$ on $(\Sph)^8$.

To do so, we use in Section~\ref{sec:minorants} the three-point semidefinite programming bound of Cohn and Woo \cite[Theorem~7]{CW}. In our setting, this bound shows that $X(a_s)$ minimizes $E_{H_s}$ whenever a certain polynomial identity in three variables holds, in which the coefficients are positive semidefinite matrices (see Proposition~\ref{prop:certificate}). \ca{For $0\leq s\leq124$, we determine such matrices numerically and verify the identity and the positive semidefiniteness with computer assistance, uniformly on $111$ intervals that cover $[0,124]$.} This verification is the content of Section~\ref{sec:nonnegative}.

The equality case is treated in Section~\ref{sec:equality}, following \cite[Section~6]{CohK07} and \cite[Section~4]{CW}. If $E_s(Y)=E_s(X(a_s))$, then every inner product $y_i\cdot y_j$ with $i\neq j$ is an interpolation node of $H_s$, and every term of the three-point bound vanishes. In Proposition~\ref{prop:moment-unique} we show that this forces $Y$ to be congruent to $X(a_s)$, provided that the matrices in the polynomial identity satisfy two additional conditions. \ca{These conditions are verified with computer assistance together with the identity in Section~\ref{sec:nonnegative}.}

The range $-2<s<0$ is treated in Section~\ref{sec:negative}. Here, we first split off from $E_s$ a positive multiple of $\norm{\sum_iy_i}^2$, which vanishes at every square antiprism. The remaining part of the energy is treated as above, \ca{with a computer-assisted verification of the polynomial identity on five intervals that cover the range $-2<s<0$.}

Finally, the range $s\geq124$ is treated in Section~\ref{sec:tail}, where we use packing geometry and convexity. A similar strategy of combining localization with convexity appears in \cite[Section~1.3]{Sch13}. Let $P$ be the square antiprism that maximizes the minimal distance among all configurations of eight points on $\Sph$. \ca{With computer assistance, we show that every critical point of $E_s$ whose energy does not exceed $E_s(X(a_s))$ is congruent to a configuration in an explicit convex neighborhood of $P$ in suitable local coordinates, and that $E_s$ is strictly convex on this neighborhood.} Since $X(a_s)$ is a critical point of $E_s$, it follows that every minimizer of $E_s$ is congruent to $X(a_s)$.

\section{The normalized energy and the parameter \texorpdfstring{$a_s$}{a\_s}}\label{sec:geometry}

In what follows, we use the normalization
$$
 \Eh_s(Y)=
 \begin{cases} E_s(Y)-28/s,&s\ne 0,\\ E_0(Y),&s=0,
 \end{cases}
$$
which preserves minimizers and critical points. Moreover, we define the function $\phi_s:[-1,1)\to\R$ via
$$
 \phi_s(t)=
 \begin{cases}((2-2t)^{-s/2}-1)/s,&s \ne 0,\\
 -\tfrac12\log(2-2t),&s=0.
 \end{cases}
$$
If $Y=(y_1,\ldots,y_8)\in(\Sph)^8$ has pairwise distinct elements, then $\Eh_s$ can be written as
\begin{equation}\label{eq:phi}
 \Eh_s(Y)=\sum_{i<j}\phi_s(y_i\cdot y_j).
\end{equation}
Notice that for $-2<s<0$ the function $\phi_s$ extends continuously to $t=1$ with $\phi_s(1)=-1/s$, and that \eqref{eq:phi} then holds for every $Y\in(\Sph)^8$.

The next proposition justifies Definition~\ref{def:as} and collects the properties of $X(a_s)$ that we use in the proof of Theorem~\ref{thm:global}.

\begin{proposition}\label{prop:parameter}
Let $s>-2$. Then the equation $F(s,a)=0$ has exactly one solution in $(0,1)$. Moreover, $a_s$ is the unique minimizer of the function $a\mapsto E_s(X(a))$ on $(0,1)$, and $X(a_s)$ is a critical point of $E_s$ on $(\Sph)^8$.
\end{proposition}

\begin{proof}
By Table~\ref{tab:distances}, a direct computation gives
\begin{equation}\label{pf:family-derivative}
\frac{\dd}{\dd a}E_s(X(a))=\frac12F(s,a),\quad a\in(0,1).
\end{equation}
Moreover, the squared distances $d_A$ and $d_B$ are decreasing in $a$, whereas $d_C$ and $d_D$ are increasing in $a$. Since $q_s=1+s/2>0$, it follows that $F(s,\cdot)$ is strictly increasing on $(0,1)$.

We proceed by determining the behavior of $F(s,a)$ at the endpoints of $(0,1)$. To that end, put $\mu=2-\sqrt2$ and $\nu=2+\sqrt2$, so that $\mu+\nu=4$ and $\mu\nu=2$. As $a\downarrow0$, we have $d_A\to2$, $d_B\to4$, $d_C\to\mu$ and $d_D\to\nu$, and therefore
$$
\lim_{a\downarrow0}F(s,a)=16(2^{-q_s}+4^{-q_s})-8(\nu\mu^{-q_s}+\mu\nu^{-q_s}).
$$
Since the function $x\mapsto x^{-q_s}$ is strictly convex on $(0,\infty)$, Jensen's inequality with the weights $\nu/4$ and $\mu/4$ gives $\nu\mu^{-q_s}+\mu\nu^{-q_s}>4(\mu\nu/2)^{-q_s}=4$. Together with $2^{-q_s}+4^{-q_s}<2$, this shows that the limit is negative. On the other hand, as $a\uparrow1$ we have $d_A,d_B\to0$ and $d_C,d_D\to4$, which implies $F(s,a)\to+\infty$. Since $F(s,\cdot)$ is continuous and strictly increasing, it has a unique zero $a_s$ in $(0,1)$, and it is negative on $(0,a_s)$ and positive on $(a_s,1)$. By \eqref{pf:family-derivative}, it follows that $a_s$ is the unique minimizer of $a\mapsto E_s(X(a))$ on $(0,1)$.

It remains to show that $X(a_s)$ is a critical point of $E_s$. Let $\Gamma$ be the group of orthogonal transformations of $\R^3$ that is generated by the reflection $(\xi_1,\xi_2,\xi_3)\mapsto(\xi_1,-\xi_2,\xi_3)$ and by the composition of the rotation by $\pi/4$ about the third coordinate axis with the reflection $(\xi_1,\xi_2,\xi_3)\mapsto(\xi_1,\xi_2,-\xi_3)$. For every $\gamma\in\Gamma$ there exists a permutation $\pi_\gamma$ of $\{1,\ldots,8\}$, which does not depend on $a$, such that the vertices of $X(a)=(x_1,\ldots,x_8)$ satisfy $\gamma x_i=x_{\pi_\gamma(i)}$. Hence, $\Gamma$ acts on $(\Sph)^8$ via $(\gamma\cdot Y)_{\pi_\gamma(i)}=\gamma y_i$. The energy $E_s$ is invariant under this action, and every $X(a)$ is a fixed point.

Now let $Y=(y_1,\ldots,y_8)$ be a fixed point of the action. Since the permutations $\pi_\gamma$ act transitively on $\{1,\ldots,8\}$, the configuration $Y$ is determined by $y_1$. Moreover, the first generator of $\Gamma$ fixes $x_1$, and hence it fixes $y_1$, which means that the second coordinate of $y_1$ vanishes. Therefore, in a neighborhood of $X(a_s)$ the fixed points of the action are exactly the configurations $X(a)$. By \eqref{pf:family-derivative}, the restriction of $E_s$ to the set of these fixed points has a critical point at $X(a_s)$, and the principle of symmetric criticality \cite[Theorem~5.4]{Palais79} shows that $X(a_s)$ is a critical point of $E_s$.
\end{proof}

\section{Reduction of the energy bound to a polynomial identity}\label{sec:minorants}

Fix $s>-2$. In this section, we use Hermite interpolation to construct a polynomial $H_s$ with $H_s\leq\phi_s$ on $[-1,1)$, and we show that the energy bound
$$
 E_s(Y)\geq E_s(X(a_s)), \quad Y\in(\Sph)^8,
$$
follows from a polynomial identity in three variables for $H_s$. Throughout this section, the inner products $t_A,t_B,t_C,t_D$ in \eqref{eq:nodes} are evaluated at $a=a_s$, and we assume that they are pairwise distinct.

\subsection{A polynomial minorant}

The polynomial $H_s$ is obtained by Hermite interpolation of $\phi_s$ at the inner products $t_A,t_B,t_C,t_D$ and at the point $-1$. Recall that a polynomial of degree at most $d$ is uniquely determined by $d+1$ Hermite interpolation conditions at pairwise distinct points, i.e., by prescribing at each of these points its derivatives up to a certain order (see \cite[Section~2.1]{CohK07}). This justifies the following definition.

\begin{definition}\label{def:Hs}
Let $d\geq12$ be an even integer, let $m_{-1}\geq0$ be an integer, and let $m_A,m_B,m_C,m_D\geq2$ be even integers such that
$$
 m_{-1}+m_A+m_B+m_C+m_D=d+1.
$$
We define the polynomial $H_s$ as the unique polynomial of degree at most $d$ that satisfies the conditions
\begin{equation}\label{eq:hermite}
\begin{aligned}
 H_s^{(j)}(-1)&=\phi_s^{(j)}(-1),&&0\le j<m_{-1},\\
 H_s^{(j)}(t_\alpha)&=\phi_s^{(j)}(t_\alpha),&&0\le j<m_\alpha,\quad\alpha\in\{A,B,C,D\}.
\end{aligned}
\end{equation}
We call $H_s$ the \emph{Hermite interpolation polynomial} of $\phi_s$ with respect to the degree $d$ and the multiplicities $m_{-1},m_A,m_B,m_C,m_D$. Moreover, we call the points $t_A,t_B,t_C,t_D$, and in the case $m_{-1}\geq1$ also the point $-1$, the \emph{interpolation nodes} of $H_s$.
\end{definition}

For simplicity, we suppress the dependence of $H_s$ on $d$ and on the multiplicities in the notation. The following lemma shows that $H_s$ is a minorant of $\phi_s$.

\begin{lemma}\label{lem:minorant}
For every $t\in[-1,1)$ we have $H_s(t)\le\phi_s(t)$, with equality if and only if $t$ is an interpolation node of $H_s$. Moreover, if $-2<s<0$, then $H_s(1)<-1/s$.
\end{lemma}

\begin{proof}
A direct computation shows that $\phi_s^{(m)}>0$ on $[-1,1)$ for every integer $m\geq1$. Now define the polynomial
$$
\omega(t)=(t+1)^{m_{-1}}\prod_{\alpha\in\{A,B,C,D\}}(t-t_\alpha)^{m_\alpha},
$$
which has degree $d+1$. By the remainder formula for Hermite interpolation (see \cite[Lemma~2.1]{CohK07}), for every $t\in[-1,1)$ there exists $\xi\in[-1,1)$ such that
\begin{equation}\label{pf:remainder}
\phi_s(t)-H_s(t)=\frac{\phi_s^{(d+1)}(\xi)}{(d+1)!}\,\omega(t).
\end{equation}
Since the integers $m_A,m_B,m_C,m_D$ are even, we have $\omega(t)\geq0$ for every $t\in[-1,1)$, and $\omega(t)=0$ holds if and only if $t$ is an interpolation node of $H_s$. Together with $\phi_s^{(d+1)}(\xi)>0$, this proves the first claim.

For the second claim, let $-2<s<0$. Since $\phi_s^{(d+2)}>0$, the function $\phi_s^{(d+1)}$ is increasing, and therefore \eqref{pf:remainder} gives
$$
\phi_s(t)-H_s(t)\geq\frac{\phi_s^{(d+1)}(-1)}{(d+1)!}\,\omega(t),\quad t\in[-1,1).
$$
Letting $t\uparrow1$ and using $\phi_s(t)\to-1/s$ as well as $\omega(1)>0$, we obtain $-1/s-H_s(1)>0$.
\end{proof}

 For every $s>-2$ and every $Y\in(\Sph)^8$, Lemma~\ref{lem:minorant} and the conditions \eqref{eq:hermite} give
\begin{equation}\label{eq:minorant-energy}
 \Eh_s(Y)\ge E_{H_s}(Y),\quad E_{H_s}(X(a_s))=\Eh_s(X(a_s)).
\end{equation}
It remains to prove the lower bound $E_{H_s}(Y)\ge\Eh_s(X(a_s))$ for every $Y\in(\Sph)^8$.

\subsection{Positive three-point kernels}

Let $u,v,t$ be real variables, which play the role of the pairwise inner products of three points on $\Sph$. For a polynomial $p$ in $u,v,t$, we denote by $\Sym(p)$ the average of $p$ over the six permutations of its variables.

Define the polynomials $Q_k$ in $u,v,t$ recursively via
$$
 Q_0=1,\quad Q_1=t-uv,\quad
 Q_k=2(t-uv)Q_{k-1}-(1-u^2)(1-v^2)Q_{k-2},\quad k\ge2.
$$
For $0\le k\le d/2-1$, let $S_k$ be the symmetric matrix of dimension $d/2+1-k$ whose entries are the polynomials
$$
 (S_k)_{ij}=\Sym(u^iv^jQ_k),\quad0\le i,j\le d/2-k.
$$
Further, for $Y=(y_1,\ldots,y_8)\in(\Sph)^8$, we define the matrix
$$
 \mathcal M_k(Y)=\sum_{i=1}^8\sum_{j=1}^8\sum_{l=1}^8
 S_k(y_i\cdot y_j,y_i\cdot y_l,y_j\cdot y_l).
$$
Observe that this sum runs over all triples $(i,j,l)$ of indices, including those in which indices coincide. The following lemma is a consequence of the positivity of the three-point kernels of Cohn and Woo \cite[Theorem~5]{CW}.

\begin{lemma}\label{lem:moment}
For every $Y\in(\Sph)^8$ and every $k\in\{0,\ldots,d/2-1\}$, we have $\mathcal M_k(Y)\succeq0$.
\end{lemma}

\begin{proof}
Since $\mathcal M_k(Y)$ depends continuously on $Y$, it suffices to prove the claim for configurations of pairwise distinct points.

Let $T_k$ be the Chebyshev polynomial of the first kind of degree $k$. Comparing the recurrence of the polynomials $Q_k$ with the recurrence $T_k(x)=2xT_{k-1}(x)-T_{k-2}(x)$ shows that
$$
Q_k(u,v,t)=\bigl((1-u^2)(1-v^2)\bigr)^{k/2}\,T_k\biggl(\frac{t-uv}{\sqrt{(1-u^2)(1-v^2)}}\biggr)
$$
whenever $(1-u^2)(1-v^2)>0$. Up to a positive factor, $T_k$ is the Gegenbauer polynomial of degree $k$ for the circle. Hence, $S_k$ is a positive multiple of a principal submatrix of the infinite matrix $S_k^3$ which is defined in \cite[Section~3]{CW}. It is shown there by means of \cite[Theorem~5]{CW} that for every finite set $C\subseteq\Sph$ the sum of the matrices $S_k^3(x\cdot y,x\cdot z,y\cdot z)$ over all $(x,y,z)\in C^3$ is positive semidefinite, which proves the claim.
\end{proof}

Next, for $0\le k\le d/2-1$, we define the matrix $R_k$ by
\begin{equation}\label{eq:three-point-kernels}
R_k=6S_k(u,v,t)+S_k(u,u,1)+S_k(v,v,1)+S_k(t,t,1)+\frac{\delta_{k0}}{7}J,
\end{equation}
whose entries are polynomials in $u,v,t$. Moreover, for $Y=(y_1,\ldots,y_8)\in(\Sph)^8$ and a polynomial $p$ in $u,v,t$, we define
$$
\Psi_Yp=\sum_{(i,j,l)}p(y_i\cdot y_j,y_i\cdot y_l,y_j\cdot y_l),
$$
where the sum runs over all triples $(i,j,l)$ of pairwise distinct indices in $\{1,\ldots,8\}$. For a matrix whose entries are polynomials in $u,v,t$, we apply $\Psi_Y$ entrywise. The matrices $R_k$ and $\mathcal M_k(Y)$ are related as follows.

\begin{lemma}\label{lem:counting}
For every $Y\in(\Sph)^8$ and every $k\in\{0,\ldots,d/2-1\}$, we have $\Psi_YR_k=6\mathcal M_k(Y)$.
\end{lemma}

\begin{proof}
We split the sum defining $\mathcal M_k(Y)$ according to which of the indices $i,j,l$ coincide. The triples of pairwise distinct indices contribute $\Psi_YS_k$. Since the entries of $S_k$ are symmetric in $u,v,t$, each of the three cases $i=j\neq l$, $i=l\neq j$ and $j=l\neq i$ contributes the matrix
$$
D_k=\sum_{i\neq j}S_k(y_i\cdot y_j,y_i\cdot y_j,1).
$$
Finally, the recurrence of the polynomials $Q_k$ gives $Q_k(1,1,1)=\delta_{k0}$ and hence $S_k(1,1,1)=\delta_{k0}J$, so that the eight triples with $i=j=l$ contribute $8\delta_{k0}J$. This shows that
\begin{equation}\label{pf:moment-split}
\mathcal M_k(Y)=\Psi_YS_k+3D_k+8\delta_{k0}J.
\end{equation}

Next, we apply $\Psi_Y$ to the five terms in the definition of $R_k$. For every pair $(i,j)$ of distinct indices there are six indices $l\notin\{i,j\}$, and therefore each of the terms $S_k(u,u,1)$, $S_k(v,v,1)$ and $S_k(t,t,1)$ contributes $6D_k$. Since there are $8\cdot7\cdot6=336$ triples of pairwise distinct indices, the constant term contributes $48\delta_{k0}J$. Consequently, we have $\Psi_YR_k=6\Psi_YS_k+18D_k+48\delta_{k0}J$, and the claim follows from \eqref{pf:moment-split}.
\end{proof}

\subsection{The three-point bound}

The next proposition reduces a lower bound for a polynomial energy $E_H$ to a polynomial identity. It is a version of the three-point bound of Cohn and Woo \cite[Theorem~7]{CW}, and it underlies all computer-assisted statements in Section~\ref{sec:cap-fixed}.

\begin{proposition}\label{prop:certificate}
Let $H$ be a polynomial in one variable and let $c\in\R$. Suppose that there exist real symmetric matrices $F_0,\ldots,F_{d/2-1}$, where $F_k$ has the same dimensions as $R_k$, polynomial column vectors $z_1,\ldots,z_r$ in the variables $u,v,t$, real symmetric matrices $W_1,\ldots,W_r$, where the dimension of $W_j$ is the number of entries of $z_j$, and multipliers $g_1,\ldots,g_r\in\{1,1-u^2,\Delta\}$, where $\Delta=1+2uvt-u^2-v^2-t^2$, such that $F_k\succeq0$ for every $k$, $W_j\succeq0$ for every $j$, and
\begin{equation}\label{eq:certificate}
 \frac{H(u)+H(v)+H(t)}3-\frac{c}{28}
 =\sum_{k=0}^{d/2-1}\ip{F_k}{R_k}
       +\sum_{j=1}^r\Sym(g_j z_j^TW_jz_j).
\end{equation}
Then every $Y\in(\Sph)^8$ satisfies
\begin{equation}\label{eq:summed}
 12\bigl(E_H(Y)-c\bigr)
 =6\sum_{k=0}^{d/2-1}\ip{F_k}{\mathcal M_k(Y)}
       +\sum_{j=1}^r\Psi_Y\Sym(g_jz_j^TW_jz_j),
\end{equation}
and every summand on the right-hand side of \eqref{eq:summed} is nonnegative. In particular, $E_H(Y)\geq c$.
\end{proposition}

\begin{proof}
We apply $\Psi_Y$ to both sides of \eqref{eq:certificate}. Since $\Psi_Y$ maps each of the polynomials $H(u)$, $H(v)$ and $H(t)$ to $12E_H(Y)$ and the constant $c/28$ to $12c$, the left-hand side of \eqref{eq:certificate} is mapped to $12(E_H(Y)-c)$. On the right-hand side, the linearity of $\Psi_Y$ and Lemma~\ref{lem:counting} give $\Psi_Y\ip{F_k}{R_k}=\ip{F_k}{\Psi_YR_k}=6\ip{F_k}{\mathcal M_k(Y)}$, and therefore \eqref{eq:summed} follows.

It remains to show that the summands on the right-hand side of \eqref{eq:summed} are nonnegative. First, the matrices $F_k$ and $\mathcal M_k(Y)$ are positive semidefinite by assumption and by Lemma~\ref{lem:moment}, respectively, which implies $\ip{F_k}{\mathcal M_k(Y)}\geq0$. Now let $(i_1,i_2,i_3)$ be a triple of pairwise distinct indices and put $(u,v,t)=(y_{i_1}\cdot y_{i_2},y_{i_1}\cdot y_{i_3},y_{i_2}\cdot y_{i_3})$. Then we have $u,v,t\in[-1,1]$ as well as $\Delta\geq0$, since $\Delta$ is the determinant of the Gram matrix of $y_{i_1},y_{i_2},y_{i_3}$. As $\Delta$ is symmetric in $u,v,t$, every multiplier $g_j$ is nonnegative at each permutation of $(u,v,t)$. Since $W_j\succeq0$, it follows that the polynomial $\Sym(g_jz_j^TW_jz_j)$ is nonnegative at $(u,v,t)$, and summing over all triples gives $\Psi_Y\Sym(g_jz_j^TW_jz_j)\geq0$.
\end{proof}

Combining Proposition~\ref{prop:certificate} for $H=H_s$ and $c=\Eh_s(X(a_s))$ with \eqref{eq:minorant-energy} and Lemma~\ref{lem:minorant} yields the following statement.

\begin{corollary}\label{cor:certificate}
Suppose that the assumptions of Proposition~\ref{prop:certificate} are satisfied for $H=H_s$ and $c=\Eh_s(X(a_s))$. Then $E_s(Y)\ge E_s(X(a_s))$ for every $Y\in(\Sph)^8$. Moreover, if $E_s(Y)=E_s(X(a_s))$, then $y_i\cdot y_j$ is an interpolation node of $H_s$ for all $i\neq j$, and every summand on the right-hand side of \eqref{eq:summed} vanishes.
\end{corollary}

\begin{proof}
By \eqref{eq:minorant-energy} and Proposition~\ref{prop:certificate}, every $Y\in(\Sph)^8$ satisfies
\begin{equation}\label{pf:chain}
\Eh_s(Y)\geq E_{H_s}(Y)\geq\Eh_s(X(a_s)).
\end{equation}
Since $\Eh_s$ and $E_s$ differ by a constant, the first claim follows.

Now suppose that $E_s(Y)=E_s(X(a_s))$. Then equality holds in both inequalities of \eqref{pf:chain}. The second equality together with \eqref{eq:summed} shows that every summand on the right-hand side of \eqref{eq:summed} vanishes. Further, the identity \eqref{eq:phi} holds for $Y$. Indeed, if $s\geq0$, then the points of $Y$ are pairwise distinct, since $E_s(Y)$ is finite, and if $-2<s<0$, then \eqref{eq:phi} holds for every configuration with $\phi_s(1)=-1/s$. Hence, the first equality in \eqref{pf:chain} reads
$$
\sum_{i<j}\bigl(\phi_s(y_i\cdot y_j)-H_s(y_i\cdot y_j)\bigr)=0.
$$
By Lemma~\ref{lem:minorant}, every summand is nonnegative, and it vanishes only if $y_i\cdot y_j$ is an interpolation node of $H_s$.
\end{proof}

\section{A criterion for uniqueness}\label{sec:equality}

In this section, we give conditions on the matrices in \eqref{eq:certificate} under which a configuration $Y$ with $E_H(Y)=c$ is congruent to $X(a_s)$. Throughout, we fix $s>-2$, set $a=a_s$, and we retain the notation and the assumptions of Section~\ref{sec:minorants}.

Let $m_*\in\R^{d/2+1}$ be the column vector
$$
 m_*=\bigl(1+2t_A^i+t_B^i+2t_C^i+2t_D^i\bigr)_{i=0}^{d/2},
$$
and denote by $\R m_*$ the linear span of $m_*$. Among the distances from a vertex $y$ of $X(a)$ to the remaining seven vertices, there are two of type $A$, one of type $B$, two of type $C$ and two of type $D$ (see Figure~\ref{fig:types}). Hence, the entry of $m_*$ with index $i$ is equal to the power sum $\sum_{z\in X(a)}(y\cdot z)^i$ for every vertex $y$ of $X(a)$. The next proposition is the uniqueness criterion that we use in the proof of Theorem~\ref{thm:global}.

\begin{proposition}\label{prop:moment-unique}
Let $H$ be a polynomial in one variable and let $c\in\R$. Suppose that the assumptions of Proposition~\ref{prop:certificate} are satisfied for $H$ and $c$, and that the following three conditions hold:
\begin{enumerate}
\item $t_\alpha+t_\beta\ne0$ for all $\alpha,\beta\in\{A,B,C,D\}$, including the case $\alpha=\beta$.
\item $\ker F_0=\R m_*$.
\item There exists $j\in\{1,\ldots,r\}$ such that $g_j=1$, $W_j\succ0$, and the vectors $z_j(t_A,t_A,t_A)$, $z_j(t_A,t_A,t_C)$ and $z_j(t_A,t_A,t_D)$ are nonzero.
\end{enumerate}
Then every $Y\in(\Sph)^8$ which satisfies $E_H(Y)=c$ and $y_i\cdot y_j\in\{-1,t_A,t_B,t_C,t_D\}$ for all $i\neq j$ is congruent to $X(a_s)$.
\end{proposition}

\begin{proof}
Let $Y=(y_1,\ldots,y_8)\in(\Sph)^8$ satisfy $E_H(Y)=c$ and $y_i\cdot y_j\in\{-1,t_A,t_B,t_C,t_D\}$ for all $i\neq j$. Since $t_A,t_B,t_C,t_D<1$, the points of $Y$ are pairwise distinct. Moreover, by Proposition~\ref{prop:certificate} the summands on the right-hand side of \eqref{eq:summed} are nonnegative and sum to $12(E_H(Y)-c)=0$. Hence, every summand on the right-hand side of \eqref{eq:summed} vanishes.

\smallskip
\textbf{Step 1.} We show that $y_i\cdot y_j\neq-1$ for all $i\neq j$. Suppose by contradiction that $y_j=-y_i$ for some $i\neq j$ and let $l\notin\{i,j\}$. Then we have $y_l\cdot y_j=-y_l\cdot y_i$, and none of these two inner products is equal to $-1$, since $y_l\neq y_i$ and $y_l\neq y_j$. Thus, both inner products belong to $\{t_A,t_B,t_C,t_D\}$ and their sum is zero, which contradicts condition \textup{(i)}. Consequently, we have $y_i\cdot y_j\in\{t_A,t_B,t_C,t_D\}$ for all $i\neq j$.

\smallskip
\textbf{Step 2.} For a point $y$ of $Y$ and a type $\alpha\in\{A,B,C,D\}$, let $n_\alpha(y)$ be the number of points $y'$ of $Y$ with $y\cdot y'=t_\alpha$. We show that
\begin{equation}\label{pf:counts}
n_A(y)=2,\quad n_B(y)=1,\quad n_C(y)=2,\quad n_D(y)=2
\end{equation}
for every point $y$ of $Y$, and that the points of $Y$ sum to zero.

To do so, define for a point $y$ of $Y$ the vector
$$
m_y=\Bigl(\sum_{y'\in Y}(y\cdot y')^i\Bigr)_{i=0}^{d/2}.
$$
Since $Q_0=1$, we have $\mathcal M_0(Y)=\sum_{y\in Y}m_ym_y^T$. As the summand $6\ip{F_0}{\mathcal M_0(Y)}$ in \eqref{eq:summed} vanishes, it follows that
$$
\sum_{y\in Y}m_y^TF_0m_y=0.
$$
Every term of this sum is nonnegative, since $F_0\succeq0$. Hence, we have $m_y^TF_0m_y=0$ and therefore $m_y\in\ker F_0=\R m_*$ by condition \textup{(ii)}. Since the entries of $m_y$ and $m_*$ with index $0$ are both equal to $8$, we obtain $m_y=m_*$.

Comparing the entries of $m_y$ and $m_*$ with the indices $i\in\{0,1,2,3\}$ and subtracting the contribution of $y'=y$ gives
$$
\sum_{\alpha\in\{A,B,C,D\}}n_\alpha(y)\,t_\alpha^i=2t_A^i+t_B^i+2t_C^i+2t_D^i,\quad i\in\{0,1,2,3\}.
$$
Since $t_A,t_B,t_C,t_D$ are pairwise distinct, the Vandermonde matrix of this linear system is invertible, which proves \eqref{pf:counts}. Moreover, the entry of $m_*$ with index $1$ is equal to $1+2t_A+t_B+2t_C+2t_D=0$. Hence, the vector $w=\sum_{y'\in Y}y'$ satisfies $y\cdot w=0$ for every point $y$ of $Y$, and summing over $y$ gives $\norm{w}^2=0$.

\smallskip
\textbf{Step 3.} We call two points of $Y$ adjacent if their inner product is equal to $t_A$. In this step, we show that the resulting graph consists of two cycles of length four, and that opposite points of each cycle have inner product $t_B$.

Let $j$ be given as in condition \textup{(iii)}. Since the summand $\Psi_Y\Sym(z_j^TW_jz_j)$ in \eqref{eq:summed} vanishes and since $W_j\succ0$, the vector $z_j$ vanishes at $(y_{i_1}\cdot y_{i_2},y_{i_1}\cdot y_{i_3},y_{i_2}\cdot y_{i_3})$ for every triple $(i_1,i_2,i_3)$ of pairwise distinct indices. Now let $y',y''$ be the two points which are adjacent to a point $y$ of $Y$. Then we have $z_j(t_A,t_A,y'\cdot y'')=0$, and condition \textup{(iii)} together with Step 1 yields $y'\cdot y''=t_B$.

By \eqref{pf:counts}, every point of $Y$ is adjacent to exactly two points, and hence the graph is a disjoint union of cycles. It contains no cycle of length three, since two points with a common adjacent point have inner product $t_B\neq t_A$. Moreover, in a cycle of length at least five, the two points at distance two from a point $y$ along the cycle are distinct and have inner product $t_B$ with $y$, which contradicts $n_B(y)=1$. Therefore, the graph consists of two cycles of length four, and opposite points of each cycle have inner product $t_B$.

\smallskip
\textbf{Step 4.} By Step 3, we may assume after relabeling that $y_1,y_2,y_3,y_4$ are the points of one cycle in cyclic order, and that $y_5,y_6,y_7,y_8$ are the points of the other cycle in cyclic order. We show that both cycles are squares with opposite centers.

Define the vector $\bar y=\frac14(y_1+y_2+y_3+y_4)$. A calculation with $t_A=a$ and $t_B=2a-1$ shows that $\norm{\bar y}^2=a$, and that for every $i\in\{1,2,3,4\}$ the vector $y_i-\bar y$ is orthogonal to $\bar y$ and satisfies $\norm{y_i-\bar y}^2=1-a$. The same calculation gives $y_3-\bar y=-(y_1-\bar y)$ and $y_4-\bar y=-(y_2-\bar y)$ as well as $(y_1-\bar y)\cdot(y_2-\bar y)=0$. Thus, the points $y_1,y_2,y_3,y_4$ are the vertices of a square with center $\bar y$ and circumradius $\sqrt{1-a}$ in the plane through $\bar y$ which is orthogonal to $\bar y$. The same holds for $y_5,y_6,y_7,y_8$, and since the points of $Y$ sum to zero, the center of the second square is $-\bar y$.

\smallskip
\textbf{Step 5.} Recall that $r_a=\sqrt{1-a}$ and $h_a=\sqrt a$. By Step 4, we may assume after applying an orthogonal transformation that the first square has the vertices $(\pm r_a,0,h_a)$ and $(0,\pm r_a,h_a)$. Then every vertex $y$ of the second square is of the form $y=(\beta_1,\beta_2,-h_a)$ with $\beta_1^2+\beta_2^2=r_a^2$. By \eqref{pf:counts} and Step 3, the three points $y'$ of $Y$ with $y\cdot y'\in\{t_A,t_B\}$ are vertices of the second square. Hence, the inner products of $y$ with the four vertices of the first square, which are given by $\pm r_a\beta_1-a$ and $\pm r_a\beta_2-a$, belong to
$$
\{t_C,t_D\}=\Bigl\{-a+\frac{1-a}{\sqrt2},\,-a-\frac{1-a}{\sqrt2}\Bigr\}.
$$
Since $r_a^2=1-a$, this gives $\beta_1,\beta_2\in\{r_a/\sqrt2,-r_a/\sqrt2\}$. As the four vertices of the second square are distinct, they are the four points $(\pm r_a/\sqrt2,\pm r_a/\sqrt2,-h_a)$, and comparing with \eqref{eq:family} shows that $Y$ is congruent to $X(a)=X(a_s)$.
\end{proof}

For $H=H_s$ and $c=\Eh_s(X(a_s))$, Corollary~\ref{cor:certificate} and \eqref{eq:summed} show that every $Y\in(\Sph)^8$ with $E_s(Y)=E_s(X(a_s))$ satisfies the two assumptions on $Y$ in Proposition~\ref{prop:moment-unique}.

\section{Certificates and proof of Theorem \ref{thm:global}}\label{sec:cap-fixed}

The goal of this section is to prove Theorem~\ref{thm:global}. To that end, we treat the ranges $0\leq s\leq124$, $-2<s<0$ and $s\geq124$ separately. In each range, the proof relies on one computer-assisted statement.
\subsection{The range \texorpdfstring{$0\leq s\leq124$}{0 <= s <= 124}}\label{sec:nonnegative}

For $0\leq s\leq124$, we use the criteria of Sections~\ref{sec:minorants} and~\ref{sec:equality}. To verify their assumptions, let
$$
 a_\infty=\frac{2\sqrt2-1}{7}.
$$
If $a_\infty<a<1/3$, then the inner products in \eqref{eq:nodes} satisfy $t_D<t_B<0<t_C<t_A$, and the sums $t_A+t_B=3a-1$, $t_A+t_D$, $t_B+t_C$ and $t_C+t_D$ are negative. In this case, $t_A,t_B,t_C,t_D$ are pairwise distinct and satisfy condition \textup{(i)} of Proposition~\ref{prop:moment-unique}. The remaining assumptions are the content of the following statement.

\begin{proposition}\label{cap:nonnegative}
\cacolor
Let $0\leq s\leq124$. Then the following hold.
\begin{enumerate}
\item We have $a_\infty<a_s<1/3$.
\item There exist an even integer $d\in\{12,14,16,18,20\}$ and multiplicities $m_{-1}$, $m_A$, $m_B$, $m_C$, $m_D$ as in Section~\ref{sec:minorants} such that the assumptions of Proposition~\ref{prop:certificate} are satisfied for $H=H_s$ and $c=\Eh_s(X(a_s))$, and such that conditions \textup{(ii)} and \textup{(iii)} of Proposition~\ref{prop:moment-unique} hold.
\end{enumerate}
\end{proposition}

\begin{proof}
\cacolor
The proof is computer-assisted. It uses a cover of $[0,124]$ by $111$ closed intervals, on each of which the degree $d$ and the multiplicities are fixed (see Figure~\ref{fig:finite-certificate-cover}). For each interval, the certificate provides approximations of the matrices $F_k$ and $W_j$ with rational entries, and the verification program shows by means of rigorous interval estimates that for every $s$ in the interval, part \textup{(i)} holds and the approximations can be corrected to matrices as required in part \textup{(ii)}. In the repository \cite{N8Fixed}, the intervals are listed in \path{Certificates/Nonnegative_s/finite/cells.json}, the certificates are contained in \path{Certificates/Nonnegative_s/certificates/finite/}, the verification program is \path{Certificates/Nonnegative_s/finite/verify_cell.py}, and the records of the verification are contained in \path{Certificates/Nonnegative_s/verification/finite/}.
\end{proof}

\begin{figure}[htbp]
\centering
\begin{tikzpicture}[x=1cm,y=1cm,font=\footnotesize,inner sep=1pt]
  \node[anchor=base west] at (-1.32,2.73) {\textit{degree}};
  \node[anchor=base,font={\small\bfseries\boldmath}] at (6.5,3.37) {Certificate intervals on $[0,124]$};
  \draw[figureInk!18,line width=.24pt] (0.000000,.25) -- (0.000000,2.43);
  \draw[figureInk!18,line width=.24pt] (1.923689,.25) -- (1.923689,2.43);
  \draw[figureInk!18,line width=.24pt] (3.048974,.25) -- (3.048974,2.43);
  \draw[figureInk!18,line width=.24pt] (4.972663,.25) -- (4.972663,2.43);
  \draw[figureInk!18,line width=.24pt] (6.654869,.25) -- (6.654869,2.43);
  \draw[figureInk!18,line width=.24pt] (8.021637,.25) -- (8.021637,2.43);
  \draw[figureInk!18,line width=.24pt] (9.530331,.25) -- (9.530331,2.43);
  \draw[figureInk!18,line width=.24pt] (11.408892,.25) -- (11.408892,2.43);
  \draw[figureInk!18,line width=.24pt] (12.518971,.25) -- (12.518971,2.43);
  \draw[figureInk!18,line width=.24pt] (13.400000,.25) -- (13.400000,2.43);
  \node[anchor=east] at (-.28,2.24) {$12$};
  \draw[figureInk!18,line width=.24pt] (0,2.24) -- (13.4,2.24);
  \filldraw[fill=figureBlue!52!white,draw=white,line width=.13pt] (0.000000,2.1360) rectangle (0.264514,2.2330);  
  \filldraw[fill=figureBlue!52!white,draw=white,line width=.13pt] (0.264514,2.1360) rectangle (0.728139,2.2330);  
  \filldraw[fill=figureBlue!52!white,draw=white,line width=.13pt] (0.728139,2.1360) rectangle (1.125286,2.2330);  
  \filldraw[fill=figureBlue!52!white,draw=white,line width=.13pt] (1.125286,2.1360) rectangle (1.472650,2.2330);  
  \filldraw[fill=figureBlue!52!white,draw=white,line width=.13pt] (1.472650,2.1360) rectangle (1.781335,2.2330);  
  \filldraw[fill=figureBlue!52!white,draw=white,line width=.13pt] (1.781335,2.1360) rectangle (2.059096,2.2330);  
  \filldraw[fill=figureBlue!52!white,draw=white,line width=.13pt] (2.059096,2.1360) rectangle (2.311569,2.2330);  
  \filldraw[fill=figureBlue!52!white,draw=white,line width=.13pt] (2.311569,2.1360) rectangle (2.542978,2.2330);  
  \filldraw[fill=figureBlue!52!white,draw=white,line width=.13pt] (2.542978,2.1360) rectangle (2.756568,2.2330);  
  \filldraw[fill=figureBlue!52!white,draw=white,line width=.13pt] (2.756568,2.1360) rectangle (2.954887,2.2330);  
  \filldraw[fill=figureBlue!52!white,draw=white,line width=.13pt] (2.954887,2.1360) rectangle (3.139976,2.2330);  
  \filldraw[fill=figureBlue!52!white,draw=white,line width=.13pt] (3.139976,2.1360) rectangle (3.313488,2.2330);  
  \filldraw[fill=figureBlue!52!white,draw=white,line width=.13pt] (3.313488,2.1360) rectangle (3.476788,2.2330);  
  \filldraw[fill=figureBlue!52!white,draw=white,line width=.13pt] (3.476788,2.1360) rectangle (3.631011,2.2330);  
  \filldraw[fill=figureBlue!52!white,draw=white,line width=.13pt] (3.631011,2.1360) rectangle (3.777113,2.2330);  
  \filldraw[fill=figureBlue!52!white,draw=white,line width=.13pt] (3.777113,2.1360) rectangle (3.915907,2.2330);  
  \filldraw[fill=figureBlue!52!white,draw=white,line width=.13pt] (3.915907,2.1360) rectangle (4.048088,2.2330);  
  \filldraw[fill=figureBlue!52!white,draw=white,line width=.13pt] (4.048088,2.1360) rectangle (4.174260,2.2330);  
  \filldraw[fill=figureBlue!52!white,draw=white,line width=.13pt] (4.174260,2.1360) rectangle (4.294944,2.2330);  
  \filldraw[fill=figureBlue!52!white,draw=white,line width=.13pt] (4.294944,2.1360) rectangle (4.410598,2.2330);  
  \filldraw[fill=figureBlue!52!white,draw=white,line width=.13pt] (4.410598,2.1360) rectangle (4.521625,2.2330);  
  \filldraw[fill=figureBlue!52!white,draw=white,line width=.13pt] (4.521625,2.1360) rectangle (4.628380,2.2330);  
  \filldraw[fill=figureBlue!52!white,draw=white,line width=.13pt] (4.628380,2.1360) rectangle (4.731181,2.2330);  
  \filldraw[fill=figureBlue!52!white,draw=white,line width=.13pt] (4.731181,2.1360) rectangle (4.830309,2.2330);  
  \filldraw[fill=figureBlue!52!white,draw=white,line width=.13pt] (4.830309,2.1360) rectangle (4.926018,2.2330);  
  \filldraw[fill=figureBlue!52!white,draw=white,line width=.13pt] (4.926018,2.1360) rectangle (5.018537,2.2330);  
  \filldraw[fill=figureBlue!52!white,draw=white,line width=.13pt] (5.018537,2.1360) rectangle (5.108070,2.2330);  
  \filldraw[fill=figureBlue!52!white,draw=white,line width=.13pt] (5.108070,2.1360) rectangle (5.194805,2.2330);  
  \filldraw[fill=figureBlue!52!white,draw=white,line width=.13pt] (5.194805,2.1360) rectangle (5.278911,2.2330);  
  \filldraw[fill=figureBlue!52!white,draw=white,line width=.13pt] (5.278911,2.1360) rectangle (5.360544,2.2330);  
  \filldraw[fill=figureBlue!52!white,draw=white,line width=.13pt] (5.360544,2.1360) rectangle (5.439843,2.2330);  
  \filldraw[fill=figureBlue!52!white,draw=white,line width=.13pt] (5.439843,2.1360) rectangle (5.516940,2.2330);  
  \filldraw[fill=figureBlue!52!white,draw=white,line width=.13pt] (5.516940,2.1360) rectangle (5.591952,2.2330);  
  \filldraw[fill=figureBlue!52!white,draw=white,line width=.13pt] (5.591952,2.1360) rectangle (5.664991,2.2330);  
  \filldraw[fill=figureBlue!52!white,draw=white,line width=.13pt] (5.664991,2.1360) rectangle (5.736156,2.2330);  
  \filldraw[fill=figureBlue!52!white,draw=white,line width=.13pt] (5.736156,2.1360) rectangle (5.805542,2.2330);  
  \filldraw[fill=figureBlue!52!white,draw=white,line width=.13pt] (5.805542,2.1360) rectangle (5.873236,2.2330);  
  \filldraw[fill=figureBlue!52!white,draw=white,line width=.13pt] (5.873236,2.1360) rectangle (5.939317,2.2330);  
  \filldraw[fill=figureBlue!52!white,draw=white,line width=.13pt] (5.939317,2.1360) rectangle (6.003862,2.2330);  
  \filldraw[fill=figureBlue!52!white,draw=white,line width=.13pt] (6.003862,2.1360) rectangle (6.066939,2.2330);  
  \filldraw[fill=figureBlue!52!white,draw=white,line width=.13pt] (6.066939,2.1360) rectangle (6.128615,2.2330);  
  \filldraw[fill=figureBlue!52!white,draw=white,line width=.13pt] (6.128615,2.1360) rectangle (6.188950,2.2330);  
  \filldraw[fill=figureBlue!52!white,draw=white,line width=.13pt] (6.188950,2.1360) rectangle (6.248001,2.2330);  
  \filldraw[fill=figureBlue!52!white,draw=white,line width=.13pt] (6.248001,2.1360) rectangle (6.305822,2.2330);  
  \filldraw[fill=figureBlue!52!white,draw=white,line width=.13pt] (6.305822,2.1360) rectangle (6.362463,2.2330);  
  \filldraw[fill=figureBlue!52!white,draw=white,line width=.13pt] (6.362463,2.1360) rectangle (6.417970,2.2330);  
  \filldraw[fill=figureBlue!52!white,draw=white,line width=.13pt] (6.417970,2.1360) rectangle (6.472390,2.2330);  
  \filldraw[fill=figureBlue!52!white,draw=white,line width=.13pt] (6.472390,2.1360) rectangle (6.525762,2.2330);  
  \filldraw[fill=figureBlue!52!white,draw=white,line width=.13pt] (6.525762,2.1360) rectangle (6.578128,2.2330);  
  \filldraw[fill=figureBlue!52!white,draw=white,line width=.13pt] (6.578128,2.1360) rectangle (6.629524,2.2330);  
  \filldraw[fill=figureBlue!52!white,draw=white,line width=.13pt] (6.629524,2.1360) rectangle (6.679985,2.2330);  
  \filldraw[fill=figureBlue!52!white,draw=white,line width=.13pt] (6.679985,2.1360) rectangle (6.729545,2.2330);  
  \filldraw[fill=figureBlue!52!white,draw=white,line width=.13pt] (6.729545,2.1360) rectangle (6.778236,2.2330);  
  \filldraw[fill=figureBlue!52!white,draw=white,line width=.13pt] (6.778236,2.1360) rectangle (6.826087,2.2330);  
  \filldraw[fill=figureBlue!52!white,draw=white,line width=.13pt] (6.826087,2.1360) rectangle (6.873127,2.2330);  
  \filldraw[fill=figureBlue!52!white,draw=white,line width=.13pt] (6.873127,2.1360) rectangle (6.919383,2.2330);  
  \filldraw[fill=figureBlue!52!white,draw=white,line width=.13pt] (6.919383,2.1360) rectangle (6.964881,2.2330);  
  \filldraw[fill=figureBlue!52!white,draw=white,line width=.13pt] (6.964881,2.1360) rectangle (7.009645,2.2330);  
  \filldraw[fill=figureBlue!52!white,draw=white,line width=.13pt] (7.009645,2.1360) rectangle (7.053698,2.2330);  
  \filldraw[fill=figureBlue!52!white,draw=white,line width=.13pt] (7.053698,2.1360) rectangle (7.075465,2.2330);  
  \filldraw[fill=figureBlue!52!white,draw=white,line width=.13pt] (7.064603,2.2470) rectangle (7.171358,2.3440);  
  \filldraw[fill=figureBlue!52!white,draw=white,line width=.13pt] (7.075465,2.1360) rectangle (7.097063,2.2330);  
  \filldraw[fill=figureBlue!52!white,draw=white,line width=.13pt] (7.097063,2.1360) rectangle (7.118494,2.2330);  
  \filldraw[fill=figureBlue!52!white,draw=white,line width=.13pt] (7.171358,2.1360) rectangle (7.274159,2.2330);  
  \filldraw[fill=figureBlue!52!white,draw=white,line width=.13pt] (7.274159,2.1360) rectangle (7.373287,2.2330);  
  \filldraw[fill=figureBlue!52!white,draw=white,line width=.13pt] (7.373287,2.1360) rectangle (7.468996,2.2330);  
  \filldraw[fill=figureBlue!52!white,draw=white,line width=.13pt] (7.468996,2.1360) rectangle (7.561515,2.2330);  
  \filldraw[fill=figureBlue!52!white,draw=white,line width=.13pt] (7.561515,2.1360) rectangle (7.651048,2.2330);  
  \filldraw[fill=figureBlue!52!white,draw=white,line width=.13pt] (7.651048,2.1360) rectangle (7.737783,2.2330);  
  \filldraw[fill=figureBlue!52!white,draw=white,line width=.13pt] (7.737783,2.1360) rectangle (7.821890,2.2330);  
  \filldraw[fill=figureBlue!52!white,draw=white,line width=.13pt] (7.821890,2.1360) rectangle (7.903522,2.2330);  
  \filldraw[fill=figureBlue!52!white,draw=white,line width=.13pt] (7.903522,2.1360) rectangle (7.943455,2.2330);  
  \filldraw[fill=figureBlue!52!white,draw=white,line width=.13pt] (7.943455,2.1360) rectangle (7.959268,2.2330);  
  \filldraw[fill=figureBlue!52!white,draw=white,line width=.13pt] (7.959268,2.1360) rectangle (7.974993,2.2330);  
  \filldraw[fill=figureBlue!52!white,draw=white,line width=.13pt] (7.974993,2.1360) rectangle (7.990628,2.2330);  
  \filldraw[fill=figureBlue!52!white,draw=white,line width=.13pt] (7.990628,2.1360) rectangle (8.006176,2.2330);  
  \filldraw[fill=figureBlue!52!white,draw=white,line width=.13pt] (8.006176,2.1360) rectangle (8.021637,2.2330);  
  \filldraw[fill=figureBlue!52!white,draw=white,line width=.13pt] (8.017780,2.2470) rectangle (8.025489,2.3440);  
  \node[anchor=east] at (-.28,1.85) {$14$};
  \draw[figureInk!18,line width=.24pt] (0,1.85) -- (13.4,1.85);
  \filldraw[fill=figureBlue!70!white,draw=white,line width=.13pt] (8.021637,1.7460) rectangle (8.314044,1.8430);  
  \filldraw[fill=figureBlue!70!white,draw=white,line width=.13pt] (8.243780,1.8570) rectangle (8.640927,1.9540);  
  \filldraw[fill=figureBlue!70!white,draw=white,line width=.13pt] (8.640927,1.7460) rectangle (8.988291,1.8430);  
  \filldraw[fill=figureBlue!70!white,draw=white,line width=.13pt] (8.988291,1.7460) rectangle (9.296976,1.8430);  
  \filldraw[fill=figureBlue!70!white,draw=white,line width=.13pt] (9.296976,1.7460) rectangle (9.574737,1.8430);  
  \filldraw[fill=figureBlue!70!white,draw=white,line width=.13pt] (9.574737,1.7460) rectangle (9.827210,1.8430);  
  \filldraw[fill=figureBlue!70!white,draw=white,line width=.13pt] (9.827210,1.7460) rectangle (10.058619,1.8430);  
  \filldraw[fill=figureBlue!70!white,draw=white,line width=.13pt] (10.058619,1.7460) rectangle (10.272209,1.8430);  
  \filldraw[fill=figureBlue!70!white,draw=white,line width=.13pt] (10.272209,1.7460) rectangle (10.470528,1.8430);  
  \filldraw[fill=figureBlue!70!white,draw=white,line width=.13pt] (10.470528,1.7460) rectangle (10.655617,1.8430);  
  \filldraw[fill=figureBlue!70!white,draw=white,line width=.13pt] (10.655617,1.7460) rectangle (10.829129,1.8430);  
  \filldraw[fill=figureBlue!70!white,draw=white,line width=.13pt] (10.829129,1.7460) rectangle (10.992429,1.8430);  
  \node[anchor=east] at (-.28,1.46) {$16$};
  \draw[figureInk!18,line width=.24pt] (0,1.46) -- (13.4,1.46);
  \filldraw[fill=figureBlue!88!white,draw=white,line width=.13pt] (10.992429,1.2790) rectangle (11.146652,1.3590);  
  \filldraw[fill=figureBlue!88!white,draw=white,line width=.13pt] (11.146652,1.2790) rectangle (11.292754,1.3590);  
  \filldraw[fill=figureBlue!88!white,draw=white,line width=.13pt] (11.292754,1.2790) rectangle (11.431548,1.3590);  
  \filldraw[fill=figureBlue!88!white,draw=white,line width=.13pt] (11.363018,1.3730) rectangle (11.627532,1.4530);  
  \filldraw[fill=figureBlue!88!white,draw=white,line width=.13pt] (11.627532,1.2790) rectangle (11.869015,1.3590);  
  \filldraw[fill=figureBlue!88!white,draw=white,line width=.13pt] (11.700161,1.3730) rectangle (11.954415,1.4530);  
  \filldraw[fill=figureBlue!88!white,draw=white,line width=.13pt] (11.710383,1.4670) rectangle (11.945055,1.5470);  
  \filldraw[fill=figureBlue!88!white,draw=white,line width=.13pt] (11.830199,1.5610) rectangle (11.982308,1.6410);  
  \filldraw[fill=figureBlue!88!white,draw=white,line width=.13pt] (11.982308,1.2790) rectangle (12.126511,1.3590);  
  \filldraw[fill=figureBlue!88!white,draw=white,line width=.13pt] (12.126511,1.2790) rectangle (12.263591,1.3590);  
  \filldraw[fill=figureBlue!88!white,draw=white,line width=.13pt] (12.263591,1.2790) rectangle (12.394217,1.3590);  
  \filldraw[fill=figureBlue!88!white,draw=white,line width=.13pt] (12.394217,1.2790) rectangle (12.457295,1.3590);  
  \filldraw[fill=figureBlue!88!white,draw=white,line width=.13pt] (12.457295,1.2790) rectangle (12.518971,1.3590);  
  \node[anchor=east] at (-.28,0.97) {$18$};
  \draw[figureInk!18,line width=.24pt] (0,0.97) -- (13.4,0.97);
  \filldraw[fill=figureBlue!75!black,draw=white,line width=.13pt] (12.518971,0.8660) rectangle (12.638357,0.9630);  
  \filldraw[fill=figureBlue!75!black,draw=white,line width=.13pt] (12.638357,0.8660) rectangle (12.752818,0.9630);  
  \filldraw[fill=figureBlue!75!black,draw=white,line width=.13pt] (12.752818,0.8660) rectangle (12.862745,0.9630);  
  \filldraw[fill=figureBlue!75!black,draw=white,line width=.13pt] (12.862745,0.8660) rectangle (12.968483,0.9630);  
  \filldraw[fill=figureBlue!75!black,draw=white,line width=.13pt] (12.942424,0.9770) rectangle (13.095231,1.0740);  
  \node[anchor=east] at (-.28,0.57) {$20$};
  \draw[figureInk!18,line width=.24pt] (0,0.57) -- (13.4,0.57);
  \filldraw[fill=figureBlue!42!black,draw=white,line width=.13pt] (13.095231,0.4660) rectangle (13.240062,0.5630);  
  \filldraw[fill=figureBlue!42!black,draw=white,line width=.13pt] (13.240062,0.4660) rectangle (13.377708,0.5630);  
  \filldraw[fill=figureBlue!42!black,draw=white,line width=.13pt] (13.309739,0.5770) rectangle (13.400000,0.6740);  
  \draw[figureInk,line width=.45pt] (0,.25) -- (13.4,.25);
  \draw[figureInk,line width=.40pt] (0.000000,.25) -- (0.000000,.16) node[below,text=black] {$0$};
  \draw[figureInk,line width=.40pt] (1.923689,.25) -- (1.923689,.16) node[below,text=black] {$1$};
  \draw[figureInk,line width=.40pt] (3.048974,.25) -- (3.048974,.16) node[below,text=black] {$2$};
  \draw[figureInk,line width=.40pt] (4.972663,.25) -- (4.972663,.16) node[below,text=black] {$5$};
  \draw[figureInk,line width=.40pt] (6.654869,.25) -- (6.654869,.16) node[below,text=black] {$10$};
  \draw[figureInk,line width=.40pt] (8.021637,.25) -- (8.021637,.16) node[below,text=black] {$17$};
  \draw[figureInk,line width=.40pt] (9.530331,.25) -- (9.530331,.16) node[below,text=black] {$30$};
  \draw[figureInk,line width=.40pt] (11.408892,.25) -- (11.408892,.16) node[below,text=black] {$60$};
  \draw[figureInk,line width=.40pt] (12.518971,.25) -- (12.518971,.16) node[below,text=black] {$90$};
  \draw[figureInk,line width=.40pt] (13.400000,.25) -- (13.400000,.16) node[below,text=black] {$124$};
  \node[anchor=west] at (14.07,.25) {$s$};
  \node[anchor=base,font={\small\bfseries\boldmath}] at (6.5,-1.05) {Detail near $s=17$};
  \begin{scope}[yshift=-9mm]
  \node[anchor=base west] at (-1.32,-.43) {\textit{degree}};
  \node[anchor=east] at (-.28,-0.92) {$12$};
  \draw[figureInk!18,line width=.24pt] (0,-0.92) -- (13.4,-0.92);
  \filldraw[fill=figureBlue!52!white,draw=white,line width=.13pt] (0.000000,-1.0210) rectangle (6.700000,-0.9270);
  \filldraw[fill=figureBlue!52!white,draw=white,line width=.13pt] (3.350000,-0.9130) rectangle (10.050000,-0.8190);
  \node[anchor=east] at (-.28,-1.24) {$14$};
  \draw[figureInk!18,line width=.24pt] (0,-1.24) -- (13.4,-1.24);
  \filldraw[fill=figureBlue!70!white,draw=white,line width=.13pt] (6.700000,-1.2870) rectangle (13.400000,-1.1930);
  \draw[figureInk,line width=.45pt] (0,-1.47) -- (13.4,-1.47);
  \draw[figureInk,line width=.4pt] (0.000000,-1.47) -- (0.000000,-1.55) node[below,text=black] {$16.95$};
  \draw[figureInk,line width=.4pt] (3.350000,-1.47) -- (3.350000,-1.55) node[below,text=black] {$16.975$};
  \draw[figureInk,line width=.4pt] (6.700000,-1.47) -- (6.700000,-1.55) node[below,text=black] {$17$};
  \draw[figureInk,line width=.4pt] (10.050000,-1.47) -- (10.050000,-1.55) node[below,text=black] {$17.025$};
  \draw[figureInk,line width=.4pt] (13.400000,-1.47) -- (13.400000,-1.55) node[below,text=black] {$17.05$};
  \node[anchor=west] at (14.07,-1.47) {$s$};
  \end{scope}
\end{tikzpicture}
\caption{The $111$ certificate intervals for $0\leq s\leq124$, grouped
by the degree $d$ of the polynomial $H_s$. Each bar in the upper plot
represents one closed interval. Its vertical position within a row
serves only to separate overlapping intervals. The upper plot uses the
coordinate $\log(1+s)$ with ticks labeled by $s$. The lower plot
shows the transition near $s=17$ on a linear scale.}
\label{fig:finite-certificate-cover}
\end{figure}

Proposition~\ref{cap:nonnegative} together with Corollary~\ref{cor:certificate} and Proposition~\ref{prop:moment-unique} has the following consequence, which includes the logarithmic energy ($s=0$) and the Coulomb energy ($s=1$).

\begin{corollary}\label{cor:nonnegative}
Let $0\leq s\leq124$. Then every $Y\in(\Sph)^8$ satisfies $E_s(Y)\geq E_s(X(a_s))$, with equality if and only if $Y$ is congruent to $X(a_s)$.
\end{corollary}

\begin{proof}
Since $a_\infty<a_s<1/3$ by Proposition~\ref{cap:nonnegative}\,\textup{(i)}, the inner products $t_A,t_B,t_C,t_D$, evaluated at $a=a_s$, are pairwise distinct and satisfy condition \textup{(i)} of Proposition~\ref{prop:moment-unique}. Let $d$ and the multiplicities $m_{-1},m_A,m_B,m_C,m_D$ be given as in Proposition~\ref{cap:nonnegative}\,\textup{(ii)}. Then Corollary~\ref{cor:certificate} gives $E_s(Y)\geq E_s(X(a_s))$ for every $Y\in(\Sph)^8$.

Now suppose that $E_s(Y)=E_s(X(a_s))$. By Corollary~\ref{cor:certificate} and \eqref{eq:summed}, we have $E_{H_s}(Y)=\Eh_s(X(a_s))$ and $y_i\cdot y_j\in\{-1,t_A,t_B,t_C,t_D\}$ for all $i\neq j$. Hence, Proposition~\ref{prop:moment-unique} shows that $Y$ is congruent to $X(a_s)$. Conversely, congruent configurations have the same energy.
\end{proof}

\subsection{The range \texorpdfstring{$-2<s<0$}{-2 < s < 0}}\label{sec:negative}

For $-2<s<0$, the kernel $K_s$ in \eqref{eq:energy} is finite on the diagonal, so that configurations with repeated points have finite energy. The following lemma shows that such configurations are not minimizers.

\begin{lemma}\label{lem:no-collisions}
Let $-2<s<0$. Then every minimizer of $E_s$ on $(\Sph)^8$ consists of pairwise distinct points.
\end{lemma}

\begin{proof}
Put $\beta=-s\in(0,2)$, so that $K_s(x,y)=-\norm{x-y}^\beta/\beta$. Suppose by contradiction that $Y=(y_1,\ldots,y_8)$ is a minimizer of $E_s$ with $y_1=y_2$. Choose a unit vector $e\in\R^3$ with $e\cdot y_1=0$ and define for $\varepsilon\in(0,\pi/2)$ the points
$$
y_\pm(\varepsilon)=\cos(\varepsilon)\,y_1\pm\sin(\varepsilon)\,e
$$
of $\Sph$. Further, let $Y_\varepsilon$ be the configuration which is obtained from $Y$ by replacing $y_1$ and $y_2$ by $y_+(\varepsilon)$ and $y_-(\varepsilon)$.

Since $\norm{y_+(\varepsilon)-y_-(\varepsilon)}=2\sin\varepsilon$, the contribution of the pair $\{1,2\}$ to the energy changes from $0$ to $-(2\sin\varepsilon)^\beta/\beta$. Now let $i\geq3$. If $y_i=y_1$, then the contribution of the pairs $\{1,i\}$ and $\{2,i\}$ does not increase, since $K_s\leq0$. If $y_i\neq y_1$, then the function $\varepsilon\mapsto K_s(y_+(\varepsilon),y_i)+K_s(y_-(\varepsilon),y_i)$ is smooth and even in a neighborhood of zero, and hence the contribution of these two pairs changes by $O(\varepsilon^2)$. Consequently, we have
$$
E_s(Y_\varepsilon)-E_s(Y)\leq-\frac{(2\sin\varepsilon)^\beta}{\beta}+O(\varepsilon^2).
$$
Since $\beta<2$, the right-hand side is negative for all sufficiently small $\varepsilon$, which gives a contradiction.
\end{proof}

Now let $q=1+s/2\in(0,1)$ and define the function $U_q:[0,1)\to\R$ via
\begin{equation}\label{eq:negative-potential}
 U_q(x)=\frac{1-(1-q)x-(1-x)^{1-q}}{q(1-q)}.
\end{equation}
For a configuration $Y=(y_1,\ldots,y_8)\in(\Sph)^8$ of pairwise distinct points, we set
$$
 \mathcal U_q(Y)
 =\sum_{i<j}U_q\left(\frac{1+y_i\cdot y_j}{2}\right).
$$
The purpose of the next lemma is to express $E_s$ in terms of $\mathcal U_q$.

\begin{lemma}\label{lem:negative-decomposition}
Let $q\in(0,1)$ and let $Y=(y_1,\ldots,y_8)\in(\Sph)^8$ consist of pairwise distinct points. Then $E_{2q-2}(Y)$ and $\gamma_q$ are given by
\begin{equation}\label{eq:negative-decomposition}
 E_{2q-2}(Y)
 =\gamma_q+\frac{4^{-q}}2\left\|\sum_{i=1}^8 y_i\right\|^2
      +2q\,4^{-q}\,\mathcal U_q(Y),
 \quad
 \gamma_q=\frac{28\cdot4^{1-q}}{2q-2}+24\cdot4^{-q}.
\end{equation}
\end{lemma}

\begin{proof}
A direct computation which uses the definition of $U_q$ and $\norm{y_i-y_j}^2=2-2y_i\cdot y_j$ shows that \eqref{eq:negative-decomposition} is equivalent to the identity $\sum_{i<j}y_i\cdot y_j=\frac12\norm{\sum_{i=1}^8y_i}^2-4$. This identity follows by expanding the squared norm.
\end{proof}

The second term on the right-hand side of \eqref{eq:negative-decomposition} is nonnegative, and it vanishes for $Y=X(a)$. Hence, the bound $E_s(Y)\geq E_s(X(a_s))$ follows from the bound $\mathcal U_q(Y)\geq\mathcal U_q(X(a_s))$, and equality in the former bound implies equality in the latter. We prove the latter bound by means of Proposition~\ref{prop:certificate}.

To that end, define $\psi_q(t)=U_q((1+t)/2)$ for $t\in[-1,1)$, so that $\mathcal U_q(Y)=\sum_{i<j}\psi_q(y_i\cdot y_j)$. In analogy to Definition~\ref{def:Hs}, we interpolate the function $\psi_q$ by a polynomial.

\begin{definition}\label{def:Jq}
Let $q\in(0,1)$, let $s=2q-2$, and suppose that the inner products $t_A,t_B,t_C,t_D$ in \eqref{eq:nodes}, evaluated at $a=a_s$, are pairwise distinct. We define the polynomial $J_q$ as the unique polynomial of degree at most $12$ that satisfies the conditions
$$
\begin{aligned}
 J_q^{(j)}(-1)&=\psi_q^{(j)}(-1),&&0\le j\le4,\\
 J_q^{(j)}(t_\alpha)&=\psi_q^{(j)}(t_\alpha),&&0\le j\le1,\quad\alpha\in\{A,B,C,D\}.
\end{aligned}
$$
We call $J_q$ the \emph{Hermite interpolation polynomial} of $\psi_q$.
\end{definition}

As in Lemma~\ref{lem:minorant}, the polynomial $J_q$ is a minorant of $\psi_q$.

\begin{lemma}\label{lem:negative-minorant}
Let $q$ and $J_q$ be given as in Definition~\ref{def:Jq}. Then for every $t\in[-1,1)$ we have $J_q(t)\leq\psi_q(t)$, with equality if and only if $t\in\{-1,t_A,t_B,t_C,t_D\}$.
\end{lemma}

\begin{proof}
Since $U_q''(x)=(1-x)^{-q-1}$, we have $U_q^{(m)}>0$ on $[0,1)$ for every $m\geq2$, and hence $\psi_q^{(m)}>0$ on $[-1,1)$ for every $m\geq2$. By the remainder formula for Hermite interpolation (see \cite[Lemma~2.1]{CohK07}), for every $t\in[-1,1)$ there exists $\xi\in[-1,1)$ such that
$$
\psi_q(t)-J_q(t)=\frac{\psi_q^{(13)}(\xi)}{13!}\,(t+1)^5\prod_{\alpha\in\{A,B,C,D\}}(t-t_\alpha)^2,
$$
and the claim follows.
\end{proof}

The assumptions of Propositions~\ref{prop:certificate} and~\ref{prop:moment-unique} for the polynomial $J_q$ are verified by the following statement.

\begin{proposition}\label{cap:negative}
\cacolor
Let $q\in(0,1)$ and $s=2q-2$. Then the following hold.
\begin{enumerate}
\item The inner products $t_A,t_B,t_C,t_D$, evaluated at $a=a_s$, are pairwise distinct and satisfy condition \textup{(i)} of Proposition~\ref{prop:moment-unique}.
\item The assumptions of Proposition~\ref{prop:certificate} are satisfied for $d=12$, the polynomial $H=J_q$ and $c=\mathcal U_q(X(a_s))$.
\item Conditions \textup{(ii)} and \textup{(iii)} of Proposition~\ref{prop:moment-unique} hold for the matrices in part \textup{(ii)}.
\end{enumerate}
\end{proposition}

\begin{proof}
\cacolor
The proof is computer-assisted and proceeds in the same way as the proof of Proposition~\ref{cap:nonnegative}, with the five intervals $[(j-1)/5,j/5]$, $1\leq j\leq5$, which cover the range of the parameter $q$. In the repository \cite{N8Fixed}, the intervals are listed in \path{Certificates/Negative_s/cells.json}, the certificates are contained in \path{Certificates/Negative_s/certificates/}, the verification program is \path{Certificates/Negative_s/verify_cell.py}, and the records of the verification are contained in \path{Certificates/Negative_s/verification/}.
\end{proof}

Proposition~\ref{cap:negative} together with Lemmas~\ref{lem:no-collisions}, \ref{lem:negative-decomposition} and~\ref{lem:negative-minorant} as well as Propositions~\ref{prop:certificate} and~\ref{prop:moment-unique} has the following consequence.

\begin{corollary}\label{cor:negative}
Let $-2<s<0$. Then every $Y\in(\Sph)^8$ satisfies $E_s(Y)\geq E_s(X(a_s))$, with equality if and only if $Y$ is congruent to $X(a_s)$.
\end{corollary}

\begin{proof}
Put $q=1+s/2$. Since $E_s$ is continuous on the compact set $(\Sph)^8$ and since congruent configurations have the same energy, it suffices to show that every minimizer of $E_s$ is congruent to $X(a_s)$.

To do so, let $Y=(y_1,\ldots,y_8)$ be a minimizer of $E_s$. By Lemma~\ref{lem:no-collisions}, the points of $Y$ are pairwise distinct. Further, by Proposition~\ref{cap:negative} we can apply Proposition~\ref{prop:certificate} to $H=J_q$ and $c=\mathcal U_q(X(a_s))$, and together with Lemma~\ref{lem:negative-minorant} we obtain
\begin{equation}\label{pf:negative-chain}
\mathcal U_q(Y)=\sum_{i<j}\psi_q(y_i\cdot y_j)\geq E_{J_q}(Y)\geq\mathcal U_q(X(a_s)).
\end{equation}
Since the vertices of $X(a_s)$ sum to zero, applying Lemma~\ref{lem:negative-decomposition} to $Y$ and to $X(a_s)$ gives
$$
E_s(Y)-E_s(X(a_s))=\frac{4^{-q}}2\Bigl\|\sum_{i=1}^8y_i\Bigr\|^2+2q\,4^{-q}\bigl(\mathcal U_q(Y)-\mathcal U_q(X(a_s))\bigr).
$$
Both terms on the right-hand side are nonnegative, whereas the left-hand side is not positive, since $Y$ is a minimizer. Hence, we have $\mathcal U_q(Y)=\mathcal U_q(X(a_s))$, and equality holds in both inequalities of \eqref{pf:negative-chain}.

The first equality in \eqref{pf:negative-chain} and Lemma~\ref{lem:negative-minorant} give $y_i\cdot y_j\in\{-1,t_A,t_B,t_C,t_D\}$ for all $i\neq j$, and the second equality reads $E_{J_q}(Y)=\mathcal U_q(X(a_s))$. Since the conditions \textup{(i)}, \textup{(ii)} and \textup{(iii)} of Proposition~\ref{prop:moment-unique} are satisfied by Proposition~\ref{cap:negative}, it follows that $Y$ is congruent to $X(a_s)$.
\end{proof}

\subsection{The range \texorpdfstring{$s\geq124$}{s >= 124}}\label{sec:tail}

For $s\geq124$, we compare the minimizers of $E_s$ with the optimal packing of eight points on $\Sph$. Recall that $a_\infty=(2\sqrt2-1)/7$ and set
$$
 P=X(a_\infty),\quad
 \tau=2(1-a_\infty)=\frac8{4+\sqrt2}.
$$
By the classical eight-point packing theorem (see \cite[Theorem~4.4]{DLM}), we have
\begin{equation}\label{eq:packing}
 \max_{Y\in(\Sph)^8}\min_{i<j}\norm{y_i-y_j}^2=\tau,
\end{equation}
and the maximum in \eqref{eq:packing} is attained exactly at the configurations that are congruent to $P$. We write $P=(p_1,\ldots,p_8)$, where the vertices are labeled in cyclic order about the vertical axis, alternating between the upper and the lower square. Further, we denote by $\mathcal C_P$ the set of the sixteen pairs $\{i,j\}$ with $\norm{p_i-p_j}^2=\tau$, which are the pairs of types $A$ and $C$ (see Figure~\ref{fig:contacts}).

\begin{figure}[htbp]
\centering
\begin{tikzpicture}[line cap=round,line join=round,font=\small]
\path[use as bounding box] (-5.25,-2.75) rectangle (6.50,2.75);
\draw[figureInk!85,line width=0.75pt] (-1.220000,0.000000) -- (-0.549390,2.050610);
\draw[figureInk!85,line width=0.75pt] (-0.549390,2.050610) -- (-2.600000,1.380000);
\draw[figureInk!85,line width=0.75pt] (-2.600000,1.380000) -- (-4.650610,2.050610);
\draw[figureInk!85,line width=0.75pt] (-4.650610,2.050610) -- (-3.980000,0.000000);
\draw[figureInk!85,line width=0.75pt] (-3.980000,0.000000) -- (-4.650610,-2.050610);
\draw[figureInk!85,line width=0.75pt] (-4.650610,-2.050610) -- (-2.600000,-1.380000);
\draw[figureInk!85,line width=0.75pt] (-2.600000,-1.380000) -- (-0.549390,-2.050610);
\draw[figureInk!85,line width=0.75pt] (-0.549390,-2.050610) -- (-1.220000,0.000000);
\draw[figureBlue,line width=1.25pt] (-1.220000,0.000000) -- (-2.600000,1.380000);
\draw[figureGreen,line width=1.25pt] (-0.549390,2.050610) -- (-4.650610,2.050610);
\draw[figureBlue,line width=1.25pt] (-2.600000,1.380000) -- (-3.980000,0.000000);
\draw[figureGreen,line width=1.25pt] (-4.650610,2.050610) -- (-4.650610,-2.050610);
\draw[figureBlue,line width=1.25pt] (-3.980000,0.000000) -- (-2.600000,-1.380000);
\draw[figureGreen,line width=1.25pt] (-4.650610,-2.050610) -- (-0.549390,-2.050610);
\draw[figureBlue,line width=1.25pt] (-2.600000,-1.380000) -- (-1.220000,0.000000);
\draw[figureGreen,line width=1.25pt] (-0.549390,-2.050610) -- (-0.549390,2.050610);
\filldraw[fill=figureBlue,draw=white,line width=0.65pt] (-1.220000,0.000000) circle[radius=3pt];
\node[anchor=east,inner sep=1pt] at (-1.490000,0.000000) {$p_{1}$};
\filldraw[fill=figureGreen,draw=white,line width=0.65pt] (-0.549390,2.050610) circle[radius=3pt];
\node[anchor=south west,inner sep=1pt] at (-0.344329,2.255671) {$p_{2}$};
\filldraw[fill=figureBlue,draw=white,line width=0.65pt] (-2.600000,1.380000) circle[radius=3pt];
\node[anchor=north,inner sep=1pt] at (-2.600000,1.110000) {$p_{3}$};
\filldraw[fill=figureGreen,draw=white,line width=0.65pt] (-4.650610,2.050610) circle[radius=3pt];
\node[anchor=south east,inner sep=1pt] at (-4.855671,2.255671) {$p_{4}$};
\filldraw[fill=figureBlue,draw=white,line width=0.65pt] (-3.980000,0.000000) circle[radius=3pt];
\node[anchor=west,inner sep=1pt] at (-3.710000,0.000000) {$p_{5}$};
\filldraw[fill=figureGreen,draw=white,line width=0.65pt] (-4.650610,-2.050610) circle[radius=3pt];
\node[anchor=north east,inner sep=1pt] at (-4.855671,-2.255671) {$p_{6}$};
\filldraw[fill=figureBlue,draw=white,line width=0.65pt] (-2.600000,-1.380000) circle[radius=3pt];
\node[anchor=south,inner sep=1pt] at (-2.600000,-1.110000) {$p_{7}$};
\filldraw[fill=figureGreen,draw=white,line width=0.65pt] (-0.549390,-2.050610) circle[radius=3pt];
\node[anchor=north west,inner sep=1pt] at (-0.344329,-2.255671) {$p_{8}$};
\node[anchor=west] at (1.00,1.53) {Between the squares};
\draw[figureInk!85,line width=0.75pt] (1.00,1.12) -- (1.78,1.12);
\node[anchor=west] at (1.98,1.12) {pairs $\{k,k+1\}$, type $C$};
\node[anchor=west] at (1.00,0.12) {Within each square};
\draw[figureBlue,line width=1.25pt] (1.00,-0.25) -- (1.78,-0.25);
\draw[figureGreen,line width=1.25pt] (1.00,-0.45) -- (1.78,-0.45);
\node[anchor=west] at (1.98,-0.35) {pairs $\{k,k+2\}$, type $A$};
\node[anchor=west] at (1.00,-1.20) {};
\node[anchor=west] at (1.00,-1.82) {$16$ pairs at squared distance $\tau$};
\end{tikzpicture}
\caption{The graph on the vertices $p_1,\ldots,p_8$ of the optimal packing $P=X(a_\infty)$ whose edges are the sixteen pairs in $\mathcal C_P$, where the indices are understood modulo $8$. The inner blue and outer green cycles represent the upper and lower squares, respectively.}
\label{fig:contacts}
\end{figure}

We set $p=s/2$ and consider the energy
$$
 \mathcal R_p(Y)=\sum_{i<j}\norm{y_i-y_j}^{-2p}=2p\,E_{2p}(Y).
$$
In order to introduce local coordinates near $P$, let
$$
 V=\left\{z=(z_1,\ldots,z_8)\in(\R^3)^8:z_i\cdot p_i=0\ \text{for}\ 1\leq i\leq8,
\ \sum_{i=1}^8 p_i\times z_i=0\right\},
 \quad \norm z^2=\sum_{i=1}^8\norm{z_i}^2.
$$
The condition $\sum_i p_i\times z_i=0$ means that $z$ is orthogonal to $(\omega\times p_1,\ldots,\omega\times p_8)$ for every $\omega\in\R^3$, i.e., to every infinitesimal rotation of $P$. For $z\in V$ with $\norm{z}<3/10$, we define the configuration $\chi_P(z)\in(\Sph)^8$ via
\begin{equation}\label{eq:slice}
 \chi_P(z)_i=\sqrt{1-\norm{z_i}^2}\,p_i+z_i,\quad 1\leq i\leq8.
\end{equation}
For each pair $e=\{i,j\}\in\mathcal C_P$, we denote by $f_e(z)=\norm{\chi_P(z)_i-\chi_P(z)_j}^2$ the squared distance between the corresponding vertices of $\chi_P(z)$, and by $\ell_e=Df_e(0)$ the derivative of $f_e$ at $z=0$, which is a linear form on $V$. Using these linear forms, we define the function $\Lambda_p:V\to\R$ via
$$
 \Lambda_p(z)=\sum_{e\in\mathcal C_P}
 \exp\left(-\frac p\tau\,\ell_e(z)\right).
$$
Finally, with the constants $\eta=3/4$, $C_0=9669627/625000$ and $\varepsilon=3\eta^2/(124\tau)$, we set
\begin{equation}\label{eq:tail-domain}
 \mathcal D_p=\left\{z\in V:\norm{z}\leq\frac{\eta}{p},\
 \Lambda_p(z)\leq C_0e^{\varepsilon}\right\}.
\end{equation}
The set $\mathcal D_p$ is convex, since $\Lambda_p$ is a sum of exponentials of linear forms. Moreover, for $p\geq62$ we have $\eta/p<3/10$, so that $\chi_P$ is defined on $\mathcal D_p$. The following statement localizes the minimizers of $\mathcal R_p$ in $\chi_P(\mathcal D_p)$.

\begin{proposition}\label{cap:tail}
\cacolor
Let $p\geq62$. Then the following hold.
\begin{enumerate}
\item We have $\tau^p\mathcal R_p(X(a_{2p}))<C_0$.
\item Every critical point $Y$ of $\mathcal R_p$ on $(\Sph)^8$ that consists of pairwise distinct points and satisfies $\tau^p\mathcal R_p(Y)\leq C_0$ is congruent to $\chi_P(z)$ for some $z\in\mathcal D_p$.
\item For every $z\in\mathcal D_p$ and every $v\in V$, the second derivative of $\mathcal R_p\circ\chi_P$ at $z$ satisfies
\begin{equation}\label{eq:tail-hessian}
 \frac{\tau^{p+2}}{p}
 D^2(\mathcal R_p\circ\chi_P)(z)[v,v]
 \geq \left(\frac34(p+1)-16\tau\right)\norm{v}^2.
\end{equation}
\end{enumerate}
\end{proposition}

\begin{proof}
\cacolor
The proof is computer-assisted. In the repository \cite{N8Fixed}, the certificates and the verification programs are contained in \path{Certificates/Nonnegative_s/tail/}, and the records of the verification are contained in \path{Certificates/Nonnegative_s/verification/tail/}. Part \textup{(ii)} relies on the sum-of-squares identity in \path{sos_gap.witness.json}, which is verified by \path{verify_sos_gap.py} and shows that every configuration whose distance to the congruence class of $P$ equals $3/10$ contains two points at squared distance less than $\tau-289/10000$. It further relies on a finite enumeration in the classification of locally optimal packings, which is carried out by \path{verify_packing.py}, and on bounds for the changes $f_e(z)-\tau$, $e\in\mathcal C_P$, of the sixteen distances in terms of $\norm z$, which are verified by \path{verify_wholeball_mean.py} and \path{verify_variance.py} with the matrix data in \path{variance_witness.json}. The remaining numerical inequalities in the proofs of parts \textup{(i)}, \textup{(ii)} and \textup{(iii)} are verified by \path{verify_tail124.py}.
\end{proof}

Since $\tau<3/2$, the right-hand side of \eqref{eq:tail-hessian} is positive for $p\geq62$ and $v\neq0$. Hence, the function $\mathcal R_p\circ\chi_P$ is strictly convex on the convex set $\mathcal D_p$, and therefore it has at most one critical point in $\mathcal D_p$. Since every minimizer of $\mathcal R_p$ is a critical point and $X(a_{2p})$ is a critical point by Proposition~\ref{prop:parameter}, Proposition~\ref{cap:tail} has the following consequence.

\begin{corollary}\label{cor:tail}
Let $s\geq124$. Then every $Y\in(\Sph)^8$ satisfies $E_s(Y)\geq E_s(X(a_s))$, with equality if and only if $Y$ is congruent to $X(a_s)$.
\end{corollary}

\begin{proof}
Put $p=s/2$. Since $E_s$ is lower semicontinuous on the compact set $(\Sph)^8$ and since congruent configurations have the same energy, it suffices to show that every minimizer of $E_s$ is congruent to $X(a_s)$.

To do so, let $Y$ be a minimizer of $E_s$. The points of $Y$ are pairwise distinct, since $E_s(Y)$ is finite. As $\mathcal R_p=2pE_s$, the configuration $Y$ is a critical point of $\mathcal R_p$ with $\mathcal R_p(Y)\leq\mathcal R_p(X(a_s))$, and $X(a_s)$ is a critical point of $\mathcal R_p$ by Proposition~\ref{prop:parameter}. Hence, parts \textup{(i)} and \textup{(ii)} of Proposition~\ref{cap:tail} show that there exist $z,z'\in\mathcal D_p$ such that $Y$ is congruent to $\chi_P(z)$ and $X(a_s)$ is congruent to $\chi_P(z')$.

Since $\mathcal R_p$ is invariant under congruence, the configurations $\chi_P(z)$ and $\chi_P(z')$ are critical points of $\mathcal R_p$, and therefore $z$ and $z'$ are critical points of $\mathcal R_p\circ\chi_P$. As this function has at most one critical point in $\mathcal D_p$, we obtain $z=z'$, which completes the proof.
\end{proof}

Theorem~\ref{thm:global} now follows by combining Corollaries~\ref{cor:nonnegative}, \ref{cor:negative} and~\ref{cor:tail}.

\section{Counterexamples for completely monotonic potentials}\label{sec:Count}

The goal of this section is to prove Theorem~\ref{thm:CM} and Theorem~\ref{thm:CM2}.

\subsection{The counterexample for eight points}

For a function $v:(0,\infty)\to\R$ and $a\in(0,1)$, the multiplicities in Table~\ref{tab:distances} give
\begin{equation}\label{eq:cm-family}
 G_v(a):=E_v(X(a))
 =8v(d_A)+4v(d_B)+8v(d_C)+8v(d_D),
\end{equation}
where the squared distances $d_A,d_B,d_C,d_D$ are evaluated at $a$. Now, we fix the constants
\begin{equation}\label{eq:cm-constants}
 a_*=\frac5{18},\quad n=10000,\quad
 d_*=\frac{13}{9},
\end{equation}
for which $d_A=d_*$ at $a=a_*$, and we define for $t>0$ the functions
\begin{equation}\label{eq:cm-potential}
 g(t)=\frac1{n+t},\quad h(t)=\left(\frac{d_*}{t}\right)^n,
 \quad f=g+\lambda h,\quad\text{where}\quad \lambda=-\frac{G_g'(a_*)}{G_h'(a_*)}.
\end{equation}
Theorem~\ref{thm:CM} is a consequence of the following statement.

\begin{proposition}\label{prop:cm-counterexample}
The number $\lambda$ in \eqref{eq:cm-potential} is well-defined and positive, and the function $f$ is strictly completely monotonic on $(0,\infty)$. Moreover, there exists a configuration $Y\in(\Sph)^8$ of pairwise distinct points such that
\begin{equation}\label{eq:cm-comparison}
 E_f(Y)<E_f(X(a)),\quad 0<a<1.
\end{equation}
\end{proposition}

The functions $g$ and $h$ in \eqref{eq:cm-potential} are not the only possible choice, as the following remark shows.

\begin{remark}\label{rem:cm-variants}
Similar constructions give counterexamples with different functional forms. With $n$, $a_*$ and $d_*$ as in \eqref{eq:cm-constants}, consider the pairs of functions
$$
\begin{array}{c|c|c}
 i & g_i(t) & h_i(t)\\ \hline
 1 & (n+t)^{-1} & (d_*/t)^n\\
 2 & e^{-t/n} & e^{-n(t-d_*)}\\
 3 & (n+t)^{-1/2} & (d_*/t)^n\\
 4 & \log(1+(n+t)^{-1}) & (d_*/t)^n\\
 5 & (4-t)^2 & \bigl((4-t)/(4-d_*)\bigr)^n
\end{array}
$$
and define
$$
 \lambda_i=-\frac{G_{g_i}'(a_*)}{G_{h_i}'(a_*)},
 \quad f_i=g_i+\lambda_i h_i.
$$
The proof of Proposition~\ref{prop:cm-counterexample} extends to each of these pairs and shows that $\lambda_i>0$ and that a small deformation of $X(a_*)$ has lower $f_i$-energy than every square antiprism. The functions $f_1,\ldots,f_4$ are strictly completely monotonic on $(0,\infty)$, while $f_5$ is a polynomial that is completely monotonic on $(0,4]$. The pair $(g_1,h_1)$ is the pair in \eqref{eq:cm-potential}.
\end{remark}

\subsection{A deformation of the antiprism}

The configuration $Y$ in Proposition~\ref{prop:cm-counterexample} is obtained from $X(a_*)$ by a deformation of the upper square. To define this deformation, let
$$
 c(u)=\frac{1-u^2}{1+u^2},\quad s(u)=\frac{2u}{1+u^2}, \quad u\in\R,
$$
and observe that $c(u)^2+s(u)^2=1$. For $0<a<1$, we put $r=\sqrt{1-a}$, $b=\sqrt a$ and $k=r/\sqrt2$. Moreover, with $c=c(u)$ and $s=s(u)$, we define the configuration $Y_a(u) \in (\Sph)^8$ via
\begin{equation}\label{eq:cm-deformation}
\begin{split}
 Y_a(u)=\big(& (rc-bs,0,bc+rs),\ (0,rc+bs,bc-rs),\\
            & (-rc+bs,0,bc+rs),\ (0,-rc-bs,bc-rs),\\
            & (k,k,-b),\ (-k,k,-b),\ (-k,-k,-b),\ (k,-k,-b)\big),
\end{split}
\end{equation}
which satisfies $Y_a(0)=X(a)$.

The first and second derivative of the energy along this deformation are expressed in terms of the following quantities. Write $\sigma=\sqrt2$ and $\kappa_\pm=1\pm1/\sqrt2$. For $v\in C^2((0,\infty))$ and $0<a<1$, we define
\begin{align}
 L_a(v)={}&-16v'(d_A)-16v'(d_B)
   +8(2+\sigma)v'(d_C)+8(2-\sigma)v'(d_D),\label{eq:cm-L}\\
 Q_a(v)={}&16\bigl(v'(d_A)+t_Bv'(d_B)+t_Cv'(d_C)+t_Dv'(d_D)\bigr)
 +128a(1-a)v''(d_B)\notag\\
 &+32a(1-a)\bigl(\kappa_+^2v''(d_C)+\kappa_-^2v''(d_D)\bigr),
 \label{eq:cm-Q}
\end{align}
where, as usual, the squared distances $d_A,d_B,d_C,d_D$ and the inner products $t_B,t_C,t_D$ in \eqref{eq:nodes} are evaluated at $a$.

\begin{lemma}\label{lem:cm-variation}
Let $0<a<1$ and $v\in C^2((0,\infty))$. Then the function $G_v$ in \eqref{eq:cm-family} satisfies $G_v'(a)=L_a(v)$ as well as
\begin{equation}\label{eq:cm-convexity}
 G_v''(a)=32v''(d_A)+64v''(d_B)
   +8(2+\sigma)^2v''(d_C)+8(2-\sigma)^2v''(d_D).
\end{equation}
Moreover, we have
$$
 \left.\frac{\dd}{\dd u}E_v(Y_a(u))\right|_{u=0}=0,
 \quad
 \left.\frac{\dd^2}{\dd u^2}E_v(Y_a(u))\right|_{u=0}=4Q_a(v).
$$
\end{lemma}

\begin{proof}
The formulas for $G_v'(a)$ and $G_v''(a)$ follow by differentiating \eqref{eq:cm-family}, since $d_A,d_B,d_C,d_D$ are affine functions of $a$ with slopes $-2$, $-4$, $2+\sigma$ and $2-\sigma$.

For the remaining two formulas, we denote for each of the $28$ pairs of points of $Y_a(u)$ by $\zeta(u)$ the squared distance of the pair. Since the last four points of $Y_a(u)$ do not depend on $u$, the function $\zeta$ is constant for the four pairs of type $A$ and the two pairs of type $B$ among these points. For the remaining $22$ pairs, we use $c(0)=1$, $s(0)=0$, $c'(0)=0$, $s'(0)=2$, $c''(0)=-4$ and $s''(0)=0$, and with $\rho=rb$ we obtain the following values.
$$
\begin{array}{c|c|c|c|c}
\text{type}&\text{number of pairs}&\zeta(0)&\zeta'(0)&\zeta''(0)\\\hline
A&4&d_A&0&16\\
B&1&d_B&16\rho&32t_B\\
B&1&d_B&-16\rho&32t_B\\
C&4&d_C&4\rho\kappa_+&8t_C\\
C&4&d_C&-4\rho\kappa_+&8t_C\\
D&4&d_D&4\rho\kappa_-&8t_D\\
D&4&d_D&-4\rho\kappa_-&8t_D
\end{array}
$$
The first derivative of $u\mapsto E_v(Y_a(u))$ at $u=0$ is the sum of $v'(\zeta(0))\zeta'(0)$ over all pairs, which vanishes. The second derivative at $u=0$ is the sum of $v''(\zeta(0))\zeta'(0)^2+v'(\zeta(0))\zeta''(0)$ over all pairs, and using $\rho^2=a(1-a)$ we obtain that this sum is equal to $4Q_a(v)$.
\end{proof}

In particular, \eqref{eq:cm-convexity} shows that $G_v$ is strictly convex on $(0,1)$ whenever $v''>0$.

\subsection{Estimates for the functions \texorpdfstring{$g$}{g} and \texorpdfstring{$h$}{h}}

For the rest of the section, we write $L=L_{a_*}$, $Q=Q_{a_*}$, and $S=Q+L$. For the proof of Proposition~\ref{prop:cm-counterexample} we require the following two lemmas, which contain estimates for the functions $g$ and $h$ in \eqref{eq:cm-potential}.

\begin{lemma}\label{lem:cm-slow}
The function $g$ satisfies $-22<n^3L(g)<-20$ and $n^3S(g)<-4$.
\end{lemma}

\begin{proof}
Throughout the proof, the squared distances and the inner products are evaluated at $a=a_*$. The functionals $L$ and $S$ are linear in $v$ and depend only on the values of $v'$ and $v''$ at the points $d_A,d_B,d_C,d_D\in(0,4]$. In $L$, the absolute values of the coefficients of $v'$ sum to $64$. In $S=Q+L$, the absolute values of the coefficients of $v'$ sum to at most $128$, since $|t_B|,|t_C|,|t_D|\leq1$, and the coefficients of $v''$ are nonnegative and sum to
$$
128a_*(1-a_*)+32a_*(1-a_*)(\kappa_+^2+\kappa_-^2)=224a_*(1-a_*)\leq56.
$$
Moreover, $L$ and $S$ vanish on affine functions, since $-16-16+8(2+\sigma)+8(2-\sigma)=0$ and $1+t_B+t_C+t_D=0$. Finally, for the function $w(t)=t^2$ a direct computation gives
\begin{equation}\label{pf:quadratic}
L(w)=384a_*-128=-\frac{64}3,\quad S(w)=-896a_*^2+1152a_*-256=-\frac{416}{81}.
\end{equation}

Now define the polynomial $g_0(t)=-t/n^2+t^2/n^3$. By Taylor's theorem, for every $t\in(0,4]$ we have the estimates
$$
|g'(t)-g_0'(t)|\leq\frac{3t^2}{n^4}\leq\frac{48}{n^4},\quad|g''(t)-g_0''(t)|\leq\frac{6t}{n^4}\leq\frac{24}{n^4}.
$$
Since $L(g_0)=L(w)/n^3$ and $S(g_0)=S(w)/n^3$, it follows from \eqref{pf:quadratic} that
$$
\Bigl|n^3L(g)+\frac{64}3\Bigr|\leq\frac{64\cdot48}{n}=\frac{3072}{n},\quad\Bigl|n^3S(g)+\frac{416}{81}\Bigr|\leq\frac{128\cdot48+56\cdot24}{n}=\frac{7488}{n}.
$$
As $n=10000$, the right-hand sides are less than $1/3$ and $1$, respectively, and the claim follows from $416/81>5$.
\end{proof}

The corresponding estimates for $h$ read as follows.

\begin{lemma}\label{lem:cm-fast}
The function $h$ satisfies $L(h)>8n/d_*>0$ and $|S(h)/L(h)|<1/10$.
\end{lemma}

\begin{proof}
At $a=a_*$ we have $d_A=d_*$, $d_B=2d_*$, $d_C=(46-13\sqrt2)/18$ and $d_D=(46+13\sqrt2)/18$. Using $\sqrt2<10/7$, we obtain $d_C-d_*>5/63$ and hence $d_*/d_C<91/96<19/20$. Since $d_B$ and $d_D$ are larger than $d_C$, the ratios $d_*/d_B$ and $d_*/d_D$ are less than $19/20$ as well. Now put $\theta=(19/20)^n$. By Bernoulli's inequality, we have $(20/19)^{20}\geq1+20/19>2$, and therefore $\theta<2^{-500}<10^{-6}$. Since
$$
h'(t)=-\frac nt\Bigl(\frac{d_*}t\Bigr)^n,\quad h''(t)=\frac{n(n+1)}{t^2}\Bigl(\frac{d_*}t\Bigr)^n,
$$
it follows that for every $\alpha\in\{B,C,D\}$ we have the bounds
\begin{equation}\label{pf:fast-bounds}
|h'(d_\alpha)|\leq\frac{n}{d_*}\,\theta,\quad|h''(d_\alpha)|\leq\frac{n(n+1)}{d_*^2}\,\theta.
\end{equation}

We first estimate $L(h)$. The term of $L(h)$ which involves $d_A$ is equal to $-16h'(d_*)=16n/d_*$, and the absolute values of the remaining coefficients of $v'$ in $L$ sum to $48$. Hence, \eqref{pf:fast-bounds} gives
$$
L(h)\geq\frac{n}{d_*}(16-48\theta)>\frac{8n}{d_*}.
$$

Next, we estimate $S(h)=Q(h)+L(h)$. The coefficients of $v'(d_A)$ in $Q$ and in $L$ cancel, and $Q$ contains no term which involves $v''(d_A)$. Since $|t_B|,|t_C|,|t_D|\leq1$, the absolute values of the coefficients of $v'(d_B)$, $v'(d_C)$ and $v'(d_D)$ in $S$ sum to at most $48+48=96$, and the coefficients of $v''$ in $S$ sum to at most $56$, as shown in the proof of Lemma~\ref{lem:cm-slow}. Using \eqref{pf:fast-bounds}, we obtain
$$
|S(h)|\leq\Bigl(\frac{96n}{d_*}+\frac{56n(n+1)}{d_*^2}\Bigr)\theta.
$$
Dividing by the lower bound for $L(h)$ and using $d_*>1$ as well as $n=10000$, we arrive at
$$
\Bigl|\frac{S(h)}{L(h)}\Bigr|<\Bigl(12+\frac{7(n+1)}{d_*}\Bigr)\theta<70019\cdot10^{-6},
$$
which is less than $1/10$.
\end{proof}

We are prepared to prove Proposition~\ref{prop:cm-counterexample}.

\begin{proof}[Proof of Proposition~\ref{prop:cm-counterexample}]
By Lemma~\ref{lem:cm-variation}, we have $G_g'(a_*)=L(g)$ and $G_h'(a_*)=L(h)$, and Lemmas~\ref{lem:cm-slow} and~\ref{lem:cm-fast} give $L(g)<0<L(h)$. Hence, the number $\lambda=-L(g)/L(h)$ is well-defined and positive. Since the functions $g$ and $h$ are strictly completely monotonic on $(0,\infty)$, the same holds for $f=g+\lambda h$.

We proceed by showing that $a_*$ minimizes the function $G_f$ on $(0,1)$. By the linearity of $L$, we have $G_f'(a_*)=L(f)=L(g)+\lambda L(h)=0$. Since $f''>0$, the identity \eqref{eq:cm-convexity} shows that $G_f$ is strictly convex on $(0,1)$, and therefore
\begin{equation}\label{pf:family-minimum}
G_f(a_*)\leq G_f(a),\quad a\in(0,1).
\end{equation}

Next, we show that $Q(f)<0$. Since $L(f)=0$, we have $Q(f)=S(f)=S(g)+\lambda S(h)$, and Lemmas~\ref{lem:cm-slow} and~\ref{lem:cm-fast} give
$$
n^3Q(f)=n^3S(g)-n^3L(g)\,\frac{S(h)}{L(h)}<-4+\frac{22}{10}<0.
$$

Finally, by Lemma~\ref{lem:cm-variation} the function $u\mapsto E_f(Y_{a_*}(u))$ has a vanishing first derivative and the negative second derivative $4Q(f)$ at $u=0$. Hence, there exists $u\neq0$ such that the points of $Y=Y_{a_*}(u)$ are pairwise distinct and such that $E_f(Y)<E_f(Y_{a_*}(0))=G_f(a_*)$. Together with \eqref{pf:family-minimum}, we obtain $E_f(Y)<G_f(a)=E_f(X(a))$ for every $a\in(0,1)$.
\end{proof}

\subsection{A counterexample for ten points on \texorpdfstring{$\mathbb S^3$}{S3}}
\label{subsec:ten-counterexample}

Recall the definition of the Petersen code $\mathcal P$ and of the orthogonal pentagon code $\mathcal O$ from Section~\ref{sec:introduction}, where $u_1,\ldots,u_5\in\mathbb S^3$ denote the vertices of a regular simplex. We compare these two codes with the configuration
$$
 Z=\{\pm u_1,\ldots,\pm u_5\}.
$$
The inner products between distinct points of the three configurations and the numbers of pairs with these inner products are recorded in Table~\ref{tab:ten}.

\begin{table}[htbp]
\centering
\small
\renewcommand{\arraystretch}{1.15}
\begin{tabular}{@{}ccc@{}}
\toprule
Configuration & Inner products & Multiplicities\\
\midrule
$Z$ & $-1,\ -1/4,\ 1/4$ & $5,\ 20,\ 20$\\
$\mathcal P$ & $-2/3,\ 1/6$ & $15,\ 30$\\
$\mathcal O$ & $(-1+\sqrt5)/4,\ (-1-\sqrt5)/4,\ 0$ & $10,\ 10,\ 25$\\
\bottomrule
\end{tabular}
\vspace{0.3cm}
\caption{Inner products and their multiplicities among the $45$ pairs of distinct points of the configurations $Z$, $\mathcal P$ and $\mathcal O$.}
\label{tab:ten}
\end{table}

Theorem~\ref{thm:CM2} is a consequence of the following statement.

\begin{proposition}\label{prop:ten}
Define the function $v:(0,4]\to\R$ via
\begin{equation}\label{eq:ten-potential}
 v(r)=\left(2-\frac r2\right)^3
      +\frac1{300}\left(2-\frac r2\right)^{14}.
\end{equation}
Then $v$ is completely monotonic on $(0,4]$, and we have
\begin{align}
 E_v(\mathcal O)-E_v(Z)
 &=\frac{1628865}{134217728},\label{eq:ten-pentagon-difference}\\
 E_v(\mathcal P)-E_v(Z)
 &=\frac{9181758824467}{213986410758144}.
 \label{eq:ten-petersen-difference}
\end{align}
In particular, $E_v(Z)<\min\{E_v(\mathcal P),E_v(\mathcal O)\}$.
\end{proposition}

\begin{proof}
Define $\varphi(t)=(1+t)^3+(1+t)^{14}/300$ for $t\in[-1,1]$, so that $v(r)=\varphi(1-r/2)$. Since all derivatives of $\varphi$ are nonnegative on $[-1,1]$, we have
$$
(-1)^kv^{(k)}(r)=2^{-k}\varphi^{(k)}(1-r/2)\geq0,\quad k\geq0,\quad r\in(0,4],
$$
i.e., $v$ is completely monotonic on $(0,4]$.

Since $v(\norm{x-y}^2)=\varphi(x\cdot y)$ for all $x,y\in\mathbb S^3$, the energy $E_v$ of each of the three configurations is the sum of the values of $\varphi$ at the inner products in Table~\ref{tab:ten}, weighted by their multiplicities. A direct computation gives
$$
E_v(Z)=\frac{98684280497}{2013265920},\quad E_v(\mathcal O)=\frac{24098807}{491520},\quad E_v(\mathcal P)=\frac{12815132203547}{261213880320},
$$
which implies the two identities in the statement.
\end{proof}

\section{Lean formalization}\label{sec:lean-formalization}

In this section, we explain the correspondence between two of our results and their formalizations in Lean~4~\cite{lean4} with its mathematical library Mathlib~\cite{mathlib}, namely the global optimality of the square antiprism for Riesz energies with $s>-2$ and the counterexample for completely monotonic potentials. All declarations in this section are taken from the file \lean{\PYG{n}{Showcase}\PYG{n+nb+bp}{.}\PYG{n}{lean}}. This file is part of the companion repository \cite{N8Fixed}. For an introduction to reading Lean statements, we refer to \cite[Section~3]{LMT26}.

\subsection{Configurations and energies}\label{subsec:lean-configurations}

The space $\R^3$ with its Euclidean norm is represented by the type
\begin{leancode}
\PYG{+w}{    }\PYG{k+kn}{abbrev}\PYG{+w}{ }\PYG{n}{Point}\PYG{+w}{ }\PYG{o}{:=}\PYG{+w}{ }\PYG{n}{EuclideanSpace}\PYG{+w}{ }\PYG{n}{ℝ}\PYG{+w}{ }\PYG{o}{(}\PYG{n}{Fin}\PYG{+w}{ }\PYG{l+m+mi}{3}\PYG{o}{)}
\end{leancode}
A family of eight points is a function \lean{\PYG{n}{Y}\PYG{+w}{ }\PYG{o}{:}\PYG{+w}{ }\PYG{n}{Fin}\PYG{+w}{ }\PYG{l+m+mi}{8}\PYG{+w}{ }\PYG{n+nb+bp}{→}\PYG{+w}{ }\PYG{n}{Point}}, where \lean{\PYG{n}{Fin}\PYG{+w}{ }\PYG{l+m+mi}{8}} is the set of indices $0,\ldots,7$. Such a family is a configuration on the sphere if its points have norm one and are pairwise distinct, i.e.,
\begin{leancode}
\PYG{+w}{    }\PYG{k+kn}{def}\PYG{+w}{ }\PYG{n}{IsConfiguration}\PYG{+w}{ }\PYG{o}{(}\PYG{n}{Y}\PYG{+w}{ }\PYG{o}{:}\PYG{+w}{ }\PYG{n}{Fin}\PYG{+w}{ }\PYG{l+m+mi}{8}\PYG{+w}{ }\PYG{n+nb+bp}{→}\PYG{+w}{ }\PYG{n}{Point}\PYG{o}{)}\PYG{+w}{ }\PYG{o}{:}\PYG{+w}{ }\PYG{k+kt}{Prop}\PYG{+w}{ }\PYG{o}{:=}
\PYG{+w}{      }\PYG{o}{(}\PYG{n+nb+bp}{∀}\PYG{+w}{ }\PYG{n}{i}\PYG{o}{,}\PYG{+w}{ }\PYG{n+nb+bp}{‖}\PYG{n}{Y}\PYG{+w}{ }\PYG{n}{i}\PYG{n+nb+bp}{‖}\PYG{+w}{ }\PYG{n+nb+bp}{=}\PYG{+w}{ }\PYG{l+m+mi}{1}\PYG{o}{)}\PYG{+w}{ }\PYG{n+nb+bp}{∧}\PYG{+w}{ }\PYG{n}{Function}\PYG{n+nb+bp}{.}\PYG{n}{Injective}\PYG{+w}{ }\PYG{n}{Y}
\end{leancode}
The energy $E_s$ in \eqref{eq:energy} is defined by
\begin{leancode}
\PYG{+w}{    }\PYG{k+kn}{def}\PYG{+w}{ }\PYG{n}{E}\PYG{+w}{ }\PYG{o}{(}\PYG{n}{s}\PYG{+w}{ }\PYG{o}{:}\PYG{+w}{ }\PYG{n}{ℝ}\PYG{o}{)}\PYG{+w}{ }\PYG{o}{(}\PYG{n}{Y}\PYG{+w}{ }\PYG{o}{:}\PYG{+w}{ }\PYG{n}{Fin}\PYG{+w}{ }\PYG{l+m+mi}{8}\PYG{+w}{ }\PYG{n+nb+bp}{→}\PYG{+w}{ }\PYG{n}{Point}\PYG{o}{)}\PYG{+w}{ }\PYG{o}{:}\PYG{+w}{ }\PYG{n}{ℝ}\PYG{+w}{ }\PYG{o}{:=}
\PYG{+w}{      }\PYG{n+nb+bp}{∑}\PYG{+w}{ }\PYG{n}{i}\PYG{+w}{ }\PYG{o}{:}\PYG{+w}{ }\PYG{n}{Fin}\PYG{+w}{ }\PYG{l+m+mi}{8}\PYG{o}{,}\PYG{+w}{ }\PYG{n+nb+bp}{∑}\PYG{+w}{ }\PYG{n}{j}\PYG{+w}{ }\PYG{o}{:}\PYG{+w}{ }\PYG{n}{Fin}\PYG{+w}{ }\PYG{l+m+mi}{8}\PYG{+w}{ }\PYG{k}{with}\PYG{+w}{ }\PYG{n}{i}\PYG{+w}{ }\PYG{n+nb+bp}{\PYGZlt{}}\PYG{+w}{ }\PYG{n}{j}\PYG{o}{,}
\PYG{+w}{        }\PYG{k}{if}\PYG{+w}{ }\PYG{n}{s}\PYG{+w}{ }\PYG{n+nb+bp}{=}\PYG{+w}{ }\PYG{l+m+mi}{0}\PYG{+w}{ }\PYG{k}{then}\PYG{+w}{ }\PYG{n+nb+bp}{\PYGZhy{}}\PYG{n}{log}\PYG{+w}{ }\PYG{n+nb+bp}{‖}\PYG{n}{Y}\PYG{+w}{ }\PYG{n}{i}\PYG{+w}{ }\PYG{n+nb+bp}{\PYGZhy{}}\PYG{+w}{ }\PYG{n}{Y}\PYG{+w}{ }\PYG{n}{j}\PYG{n+nb+bp}{‖}\PYG{+w}{ }\PYG{k}{else}\PYG{+w}{ }\PYG{n+nb+bp}{‖}\PYG{n}{Y}\PYG{+w}{ }\PYG{n}{i}\PYG{+w}{ }\PYG{n+nb+bp}{\PYGZhy{}}\PYG{+w}{ }\PYG{n}{Y}\PYG{+w}{ }\PYG{n}{j}\PYG{n+nb+bp}{‖}\PYG{+w}{ }\PYG{n+nb+bp}{\PYGZca{}}\PYG{+w}{ }\PYG{o}{(}\PYG{n+nb+bp}{\PYGZhy{}}\PYG{n}{s}\PYG{o}{)}\PYG{+w}{ }\PYG{n+nb+bp}{/}\PYG{+w}{ }\PYG{n}{s}
\end{leancode}
where the inner sum is restricted to $j>i$, so that each of the $28$ pairs of points is counted once.

\subsection{Congruence and uniqueness}\label{subsec:lean-congruence}

The congruence of two configurations is defined by
\begin{leancode}
\PYG{+w}{    }\PYG{k+kn}{def}\PYG{+w}{ }\PYG{n}{Congruent}\PYG{+w}{ }\PYG{o}{(}\PYG{n}{Y}\PYG{+w}{ }\PYG{n}{Z}\PYG{+w}{ }\PYG{o}{:}\PYG{+w}{ }\PYG{n}{Fin}\PYG{+w}{ }\PYG{l+m+mi}{8}\PYG{+w}{ }\PYG{n+nb+bp}{→}\PYG{+w}{ }\PYG{n}{Point}\PYG{o}{)}\PYG{+w}{ }\PYG{o}{:}\PYG{+w}{ }\PYG{k+kt}{Prop}\PYG{+w}{ }\PYG{o}{:=}
\PYG{+w}{      }\PYG{n+nb+bp}{∃}\PYG{+w}{ }\PYG{o}{(}\PYG{n}{U}\PYG{+w}{ }\PYG{o}{:}\PYG{+w}{ }\PYG{n}{Point}\PYG{+w}{ }\PYG{n+nb+bp}{≃ₗᵢ}\PYG{o}{[}\PYG{n}{ℝ}\PYG{o}{]}\PYG{+w}{ }\PYG{n}{Point}\PYG{o}{)}\PYG{+w}{ }\PYG{o}{(}\PYG{n}{σ}\PYG{+w}{ }\PYG{o}{:}\PYG{+w}{ }\PYG{n}{Equiv}\PYG{n+nb+bp}{.}\PYG{n}{Perm}\PYG{+w}{ }\PYG{o}{(}\PYG{n}{Fin}\PYG{+w}{ }\PYG{l+m+mi}{8}\PYG{o}{)),}
\PYG{+w}{        }\PYG{n+nb+bp}{∀}\PYG{+w}{ }\PYG{n}{i}\PYG{o}{,}\PYG{+w}{ }\PYG{n}{Y}\PYG{+w}{ }\PYG{n}{i}\PYG{+w}{ }\PYG{n+nb+bp}{=}\PYG{+w}{ }\PYG{n}{U}\PYG{+w}{ }\PYG{o}{(}\PYG{n}{Z}\PYG{+w}{ }\PYG{o}{(}\PYG{n}{σ}\PYG{+w}{ }\PYG{n}{i}\PYG{o}{))}
\end{leancode}
where \lean{\PYG{n}{Point}\PYG{+w}{ }\PYG{n+nb+bp}{≃ₗᵢ}\PYG{o}{[}\PYG{n}{ℝ}\PYG{o}{]}\PYG{+w}{ }\PYG{n}{Point}} is the type of the linear isometries of $\R^3$, i.e., of the orthogonal transformations, and \lean{\PYG{n}{Equiv}\PYG{n+nb+bp}{.}\PYG{n}{Perm}\PYG{+w}{ }\PYG{o}{(}\PYG{n}{Fin}\PYG{+w}{ }\PYG{l+m+mi}{8}\PYG{o}{)}} is the type of the permutations of the indices. The property of a configuration $Z$ to be a minimizer of $E_s$ which is unique up to congruence is expressed by
\begin{leancode}
\PYG{+w}{    }\PYG{k+kn}{def}\PYG{+w}{ }\PYG{n}{IsUniqueMinimizer}\PYG{+w}{ }\PYG{o}{(}\PYG{n}{s}\PYG{+w}{ }\PYG{o}{:}\PYG{+w}{ }\PYG{n}{ℝ}\PYG{o}{)}\PYG{+w}{ }\PYG{o}{(}\PYG{n}{Z}\PYG{+w}{ }\PYG{o}{:}\PYG{+w}{ }\PYG{n}{Fin}\PYG{+w}{ }\PYG{l+m+mi}{8}\PYG{+w}{ }\PYG{n+nb+bp}{→}\PYG{+w}{ }\PYG{n}{Point}\PYG{o}{)}\PYG{+w}{ }\PYG{o}{:}\PYG{+w}{ }\PYG{k+kt}{Prop}\PYG{+w}{ }\PYG{o}{:=}
\PYG{+w}{      }\PYG{n}{IsConfiguration}\PYG{+w}{ }\PYG{n}{Z}\PYG{+w}{ }\PYG{n+nb+bp}{∧}\PYG{+w}{ }\PYG{n+nb+bp}{∀}\PYG{+w}{ }\PYG{n}{Y}\PYG{o}{,}\PYG{+w}{ }\PYG{n}{IsConfiguration}\PYG{+w}{ }\PYG{n}{Y}\PYG{+w}{ }\PYG{n+nb+bp}{→}
\PYG{+w}{        }\PYG{n}{E}\PYG{+w}{ }\PYG{n}{s}\PYG{+w}{ }\PYG{n}{Z}\PYG{+w}{ }\PYG{n+nb+bp}{≤}\PYG{+w}{ }\PYG{n}{E}\PYG{+w}{ }\PYG{n}{s}\PYG{+w}{ }\PYG{n}{Y}\PYG{+w}{ }\PYG{n+nb+bp}{∧}\PYG{+w}{ }\PYG{o}{(}\PYG{n}{E}\PYG{+w}{ }\PYG{n}{s}\PYG{+w}{ }\PYG{n}{Y}\PYG{+w}{ }\PYG{n+nb+bp}{=}\PYG{+w}{ }\PYG{n}{E}\PYG{+w}{ }\PYG{n}{s}\PYG{+w}{ }\PYG{n}{Z}\PYG{+w}{ }\PYG{n+nb+bp}{↔}\PYG{+w}{ }\PYG{n}{Congruent}\PYG{+w}{ }\PYG{n}{Y}\PYG{+w}{ }\PYG{n}{Z}\PYG{o}{)}
\end{leancode}
Notice that $Y$ ranges over the configurations of pairwise distinct points. For $s\geq0$ this is no restriction, since a configuration with repeated points has infinite energy, and for $-2<s<0$ the configurations with repeated points are covered by Lemma~\ref{lem:no-collisions}.

\subsection{Global optimality of the square antiprism}\label{subsec:lean-antiprism}

The square antiprism $X(a)$ in \eqref{eq:family} is defined, with the same order of the vertices, via \lean{\PYG{k+kn}{def}\PYG{+w}{ }\PYG{n}{X}\PYG{+w}{ }\PYG{o}{(}\PYG{n}{a}\PYG{+w}{ }\PYG{o}{:}\PYG{+w}{ }\PYG{n}{ℝ}\PYG{o}{)}\PYG{+w}{ }\PYG{o}{:}\PYG{+w}{ }\PYG{n}{Fin}\PYG{+w}{ }\PYG{l+m+mi}{8}\PYG{+w}{ }\PYG{n+nb+bp}{→}\PYG{+w}{ }\PYG{n}{Point}}, and the function $F$ in \eqref{eq:F} is defined via \lean{\PYG{k+kn}{def}\PYG{+w}{ }\PYG{n}{F}\PYG{+w}{ }\PYG{o}{(}\PYG{n}{s}\PYG{+w}{ }\PYG{n}{a}\PYG{+w}{ }\PYG{o}{:}\PYG{+w}{ }\PYG{n}{ℝ}\PYG{o}{)}\PYG{+w}{ }\PYG{o}{:}\PYG{+w}{ }\PYG{n}{ℝ}}. The exact definitions are given in the repository. The theorem \lean{\PYG{n}{all\PYGZus{}s\PYGZus{}global\PYGZus{}optimality}} expresses the global optimality of the square antiprism. Its statement reads as follows, where we omit three clauses on the value of $a_0$ and on the minimal energies for $s=0$ and $s=1$, which are not used in this paper.
\begin{leancode}
\PYG{+w}{    }\PYG{k+kn}{theorem}\PYG{+w}{ }\PYG{n}{all\PYGZus{}s\PYGZus{}global\PYGZus{}optimality}\PYG{+w}{ }\PYG{o}{:}
\PYG{+w}{        }\PYG{n+nb+bp}{∃}\PYG{+w}{ }\PYG{n}{a}\PYG{+w}{ }\PYG{o}{:}\PYG{+w}{ }\PYG{n}{ℝ}\PYG{+w}{ }\PYG{n+nb+bp}{→}\PYG{+w}{ }\PYG{n}{ℝ}\PYG{o}{,}
\PYG{+w}{          }\PYG{o}{(}\PYG{n+nb+bp}{∀}\PYG{+w}{ }\PYG{n}{s}\PYG{+w}{ }\PYG{n+nb+bp}{\PYGZgt{}}\PYG{+w}{ }\PYG{o}{(}\PYG{n+nb+bp}{\PYGZhy{}}\PYG{l+m+mi}{2}\PYG{+w}{ }\PYG{o}{:}\PYG{+w}{ }\PYG{n}{ℝ}\PYG{o}{),}\PYG{+w}{ }\PYG{n}{a}\PYG{+w}{ }\PYG{n}{s}\PYG{+w}{ }\PYG{n+nb+bp}{∈}\PYG{+w}{ }\PYG{n}{Set}\PYG{n+nb+bp}{.}\PYG{n}{Ioo}\PYG{+w}{ }\PYG{l+m+mi}{0}\PYG{+w}{ }\PYG{l+m+mi}{1}\PYG{+w}{ }\PYG{n+nb+bp}{∧}
\PYG{+w}{            }\PYG{n+nb+bp}{∀}\PYG{+w}{ }\PYG{n}{b}\PYG{+w}{ }\PYG{n+nb+bp}{∈}\PYG{+w}{ }\PYG{n}{Set}\PYG{n+nb+bp}{.}\PYG{n}{Ioo}\PYG{+w}{ }\PYG{o}{(}\PYG{l+m+mi}{0}\PYG{+w}{ }\PYG{o}{:}\PYG{+w}{ }\PYG{n}{ℝ}\PYG{o}{)}\PYG{+w}{ }\PYG{l+m+mi}{1}\PYG{o}{,}\PYG{+w}{ }\PYG{n}{F}\PYG{+w}{ }\PYG{n}{s}\PYG{+w}{ }\PYG{n}{b}\PYG{+w}{ }\PYG{n+nb+bp}{=}\PYG{+w}{ }\PYG{l+m+mi}{0}\PYG{+w}{ }\PYG{n+nb+bp}{↔}\PYG{+w}{ }\PYG{n}{b}\PYG{+w}{ }\PYG{n+nb+bp}{=}\PYG{+w}{ }\PYG{n}{a}\PYG{+w}{ }\PYG{n}{s}\PYG{o}{)}\PYG{+w}{ }\PYG{n+nb+bp}{∧}
\PYG{+w}{          }\PYG{o}{(}\PYG{n+nb+bp}{∀}\PYG{+w}{ }\PYG{n}{s}\PYG{+w}{ }\PYG{n+nb+bp}{\PYGZgt{}}\PYG{+w}{ }\PYG{o}{(}\PYG{n+nb+bp}{\PYGZhy{}}\PYG{l+m+mi}{2}\PYG{+w}{ }\PYG{o}{:}\PYG{+w}{ }\PYG{n}{ℝ}\PYG{o}{),}\PYG{+w}{ }\PYG{n}{IsUniqueMinimizer}\PYG{+w}{ }\PYG{n}{s}\PYG{+w}{ }\PYG{o}{(}\PYG{n}{X}\PYG{+w}{ }\PYG{o}{(}\PYG{n}{a}\PYG{+w}{ }\PYG{n}{s}\PYG{o}{)))}
\end{leancode}
The first clause asserts that for every $s>-2$, the number \lean{\PYG{n}{a}\PYG{+w}{ }\PYG{n}{s}} is the unique solution of the equation $F(s,a)=0$ in $(0,1)$, i.e., \lean{\PYG{n}{a}\PYG{+w}{ }\PYG{n}{s}} represents $a_s$. The second clause is the statement of Theorem~\ref{thm:global}.

\subsection{The counterexample for completely monotonic potentials}\label{subsec:lean-counterexample}

The energy $E_v$ in \eqref{eq:cm-energy} is defined by
\begin{leancode}
\PYG{+w}{    }\PYG{k+kn}{def}\PYG{+w}{ }\PYG{n}{Ef}\PYG{+w}{ }\PYG{o}{(}\PYG{n}{v}\PYG{+w}{ }\PYG{o}{:}\PYG{+w}{ }\PYG{n}{ℝ}\PYG{+w}{ }\PYG{n+nb+bp}{→}\PYG{+w}{ }\PYG{n}{ℝ}\PYG{o}{)}\PYG{+w}{ }\PYG{o}{(}\PYG{n}{Y}\PYG{+w}{ }\PYG{o}{:}\PYG{+w}{ }\PYG{n}{Fin}\PYG{+w}{ }\PYG{l+m+mi}{8}\PYG{+w}{ }\PYG{n+nb+bp}{→}\PYG{+w}{ }\PYG{n}{Point}\PYG{o}{)}\PYG{+w}{ }\PYG{o}{:}\PYG{+w}{ }\PYG{n}{ℝ}\PYG{+w}{ }\PYG{o}{:=}
\PYG{+w}{      }\PYG{n+nb+bp}{∑}\PYG{+w}{ }\PYG{n}{i}\PYG{+w}{ }\PYG{o}{:}\PYG{+w}{ }\PYG{n}{Fin}\PYG{+w}{ }\PYG{l+m+mi}{8}\PYG{o}{,}\PYG{+w}{ }\PYG{n+nb+bp}{∑}\PYG{+w}{ }\PYG{n}{j}\PYG{+w}{ }\PYG{o}{:}\PYG{+w}{ }\PYG{n}{Fin}\PYG{+w}{ }\PYG{l+m+mi}{8}\PYG{+w}{ }\PYG{k}{with}\PYG{+w}{ }\PYG{n}{i}\PYG{+w}{ }\PYG{n+nb+bp}{\PYGZlt{}}\PYG{+w}{ }\PYG{n}{j}\PYG{o}{,}
\PYG{+w}{        }\PYG{n}{v}\PYG{+w}{ }\PYG{o}{(}\PYG{n+nb+bp}{‖}\PYG{n}{Y}\PYG{+w}{ }\PYG{n}{i}\PYG{+w}{ }\PYG{n+nb+bp}{\PYGZhy{}}\PYG{+w}{ }\PYG{n}{Y}\PYG{+w}{ }\PYG{n}{j}\PYG{n+nb+bp}{‖}\PYG{+w}{ }\PYG{n+nb+bp}{\PYGZca{}}\PYG{+w}{ }\PYG{l+m+mi}{2}\PYG{o}{)}
\end{leancode}
and the property of a function to be strictly completely monotonic on $(0,\infty)$ is defined by
\begin{leancode}
\PYG{+w}{    }\PYG{k+kn}{def}\PYG{+w}{ }\PYG{n}{StrictlyCompletelyMonotone}\PYG{+w}{ }\PYG{o}{(}\PYG{n}{v}\PYG{+w}{ }\PYG{o}{:}\PYG{+w}{ }\PYG{n}{ℝ}\PYG{+w}{ }\PYG{n+nb+bp}{→}\PYG{+w}{ }\PYG{n}{ℝ}\PYG{o}{)}\PYG{+w}{ }\PYG{o}{:}\PYG{+w}{ }\PYG{k+kt}{Prop}\PYG{+w}{ }\PYG{o}{:=}
\PYG{+w}{      }\PYG{n}{ContDiffOn}\PYG{+w}{ }\PYG{n}{ℝ}\PYG{+w}{ }\PYG{n+nb+bp}{∞}\PYG{+w}{ }\PYG{n}{v}\PYG{+w}{ }\PYG{o}{(}\PYG{n}{Set}\PYG{n+nb+bp}{.}\PYG{n}{Ioi}\PYG{+w}{ }\PYG{l+m+mi}{0}\PYG{o}{)}\PYG{+w}{ }\PYG{n+nb+bp}{∧}
\PYG{+w}{        }\PYG{n+nb+bp}{∀}\PYG{+w}{ }\PYG{n}{k}\PYG{+w}{ }\PYG{o}{:}\PYG{+w}{ }\PYG{n}{ℕ}\PYG{o}{,}\PYG{+w}{ }\PYG{n+nb+bp}{∀}\PYG{+w}{ }\PYG{n}{r}\PYG{+w}{ }\PYG{o}{:}\PYG{+w}{ }\PYG{n}{ℝ}\PYG{o}{,}\PYG{+w}{ }\PYG{l+m+mi}{0}\PYG{+w}{ }\PYG{n+nb+bp}{\PYGZlt{}}\PYG{+w}{ }\PYG{n}{r}\PYG{+w}{ }\PYG{n+nb+bp}{→}
\PYG{+w}{          }\PYG{l+m+mi}{0}\PYG{+w}{ }\PYG{n+nb+bp}{\PYGZlt{}}\PYG{+w}{ }\PYG{o}{(}\PYG{n+nb+bp}{\PYGZhy{}}\PYG{l+m+mi}{1}\PYG{+w}{ }\PYG{o}{:}\PYG{+w}{ }\PYG{n}{ℝ}\PYG{o}{)}\PYG{+w}{ }\PYG{n+nb+bp}{\PYGZca{}}\PYG{+w}{ }\PYG{n}{k}\PYG{+w}{ }\PYG{n+nb+bp}{*}\PYG{+w}{ }\PYG{n}{iteratedDeriv}\PYG{+w}{ }\PYG{n}{k}\PYG{+w}{ }\PYG{n}{v}\PYG{+w}{ }\PYG{n}{r}
\end{leancode}
where \lean{\PYG{n}{Set}\PYG{n+nb+bp}{.}\PYG{n}{Ioi}\PYG{+w}{ }\PYG{l+m+mi}{0}} denotes the interval $(0,\infty)$ and \lean{\PYG{n}{iteratedDeriv}\PYG{+w}{ }\PYG{n}{k}\PYG{+w}{ }\PYG{n}{v}\PYG{+w}{ }\PYG{n}{r}} denotes $v^{(k)}(r)$. The number $\lambda$ and the function $f$ in \eqref{eq:cm-potential} are represented by \lean{\PYG{n}{lambda}} and \lean{\PYG{n}{f}}, whose exact definitions are given in the repository. In Lean, the function \lean{\PYG{n}{f}} is defined on all of $\R$, and its values at nonpositive arguments play no role. The counterexample is expressed in the following theorem.
\begin{leancode}
\PYG{+w}{    }\PYG{k+kn}{theorem}\PYG{+w}{ }\PYG{n}{counterexample}\PYG{+w}{ }\PYG{o}{:}
\PYG{+w}{        }\PYG{l+m+mi}{0}\PYG{+w}{ }\PYG{n+nb+bp}{\PYGZlt{}}\PYG{+w}{ }\PYG{n}{lambda}\PYG{+w}{ }\PYG{n+nb+bp}{∧}
\PYG{+w}{        }\PYG{n}{StrictlyCompletelyMonotone}\PYG{+w}{ }\PYG{n}{f}\PYG{+w}{ }\PYG{n+nb+bp}{∧}
\PYG{+w}{        }\PYG{n}{Filter}\PYG{n+nb+bp}{.}\PYG{n}{Tendsto}\PYG{+w}{ }\PYG{n}{f}\PYG{+w}{ }\PYG{o}{(}\PYG{n}{nhdsWithin}\PYG{+w}{ }\PYG{l+m+mi}{0}\PYG{+w}{ }\PYG{o}{(}\PYG{n}{Set}\PYG{n+nb+bp}{.}\PYG{n}{Ioi}\PYG{+w}{ }\PYG{l+m+mi}{0}\PYG{o}{))}\PYG{+w}{ }\PYG{n}{Filter}\PYG{n+nb+bp}{.}\PYG{n}{atTop}\PYG{+w}{ }\PYG{n+nb+bp}{∧}
\PYG{+w}{        }\PYG{n+nb+bp}{∃}\PYG{+w}{ }\PYG{n}{Y}\PYG{+w}{ }\PYG{o}{:}\PYG{+w}{ }\PYG{n}{Fin}\PYG{+w}{ }\PYG{l+m+mi}{8}\PYG{+w}{ }\PYG{n+nb+bp}{→}\PYG{+w}{ }\PYG{n}{Point}\PYG{o}{,}\PYG{+w}{ }\PYG{n}{IsConfiguration}\PYG{+w}{ }\PYG{n}{Y}\PYG{+w}{ }\PYG{n+nb+bp}{∧}
\PYG{+w}{          }\PYG{n+nb+bp}{∀}\PYG{+w}{ }\PYG{n}{a}\PYG{+w}{ }\PYG{n+nb+bp}{∈}\PYG{+w}{ }\PYG{n}{Set}\PYG{n+nb+bp}{.}\PYG{n}{Ioo}\PYG{+w}{ }\PYG{o}{(}\PYG{l+m+mi}{0}\PYG{+w}{ }\PYG{o}{:}\PYG{+w}{ }\PYG{n}{ℝ}\PYG{o}{)}\PYG{+w}{ }\PYG{l+m+mi}{1}\PYG{o}{,}\PYG{+w}{ }\PYG{n+nb+bp}{∀}\PYG{+w}{ }\PYG{n}{Z}\PYG{+w}{ }\PYG{o}{:}\PYG{+w}{ }\PYG{n}{Fin}\PYG{+w}{ }\PYG{l+m+mi}{8}\PYG{+w}{ }\PYG{n+nb+bp}{→}\PYG{+w}{ }\PYG{n}{Point}\PYG{o}{,}
\PYG{+w}{            }\PYG{n}{Congruent}\PYG{+w}{ }\PYG{n}{Z}\PYG{+w}{ }\PYG{o}{(}\PYG{n}{X}\PYG{+w}{ }\PYG{n}{a}\PYG{o}{)}\PYG{+w}{ }\PYG{n+nb+bp}{→}\PYG{+w}{ }\PYG{n}{Ef}\PYG{+w}{ }\PYG{n}{f}\PYG{+w}{ }\PYG{n}{Y}\PYG{+w}{ }\PYG{n+nb+bp}{\PYGZlt{}}\PYG{+w}{ }\PYG{n}{Ef}\PYG{+w}{ }\PYG{n}{f}\PYG{+w}{ }\PYG{n}{Z}
\end{leancode}
The first, the second and the last clause form the statement of Proposition~\ref{prop:cm-counterexample}, where the last clause is formulated for every configuration $Z$ that is congruent to $X(a)$. The third clause asserts in addition that $f(t)\to+\infty$ as $t\downarrow0$.

\subsection{Source code and verification}\label{subsec:lean-source}

The file \lean{\PYG{n}{Showcase}\PYG{n+nb+bp}{.}\PYG{n}{lean}} states the two theorems of this section without their proofs. The file \lean{\PYG{n}{Showcase\PYGZus{}WithProofs}\PYG{n+nb+bp}{.}\PYG{n}{lean}} contains the same statements together with complete proofs, and its compilation with Lean~4.33.0 shows that both theorems depend only on the three standard axioms \lean{\PYG{n}{propext}}, \lean{\PYG{n}{Classical}\PYG{n+nb+bp}{.}\PYG{n}{choice}} and \lean{\PYG{n}{Quot}\PYG{n+nb+bp}{.}\PYG{n}{sound}}.

\section*{Acknowledgments}

We thank Henry Cohn for communicating the counterexample for ten points on $\mathbb S^3$ in Theorem~\ref{thm:CM2}. We also thank Joseph Tooby-Smith and Alex Zughaid for making their independent Lean formalization of the eight-point Thomson problem publicly available~\cite{TSZ26}.

Recent work has also resolved the classical Thomson problem on $\Sph$ for seven points \cite{Shen26,Tran26,ZarcoSeven26}, nine points \cite{Huang26,N9Lean26,Xie26}, ten points \cite{Graham26,Xie26} and thirty-two points \cite{Kim32}. Moreover, energy bounds for seven and eight points which are conditional on polynomial certificates are given in \cite{CTW26}.

L.L.~is grateful to the Azrieli Foundation for the award of an Azrieli Fellowship and acknowledges the support of this research by ISF Grant No.~854/25.

\end{document}